\documentclass[oneside, 12pt]{amsart}

\usepackage{amsmath, amsfonts, amssymb, amsthm, graphics}

\newtheorem{theorem}{Theorem}[section]
\newtheorem{lemma}[theorem]{Lemma}
\newtheorem{corollary}[theorem]{Corollary}
\newtheorem{claim}{Claim}[section]

\theoremstyle{definition}
\newtheorem{remark}[theorem]{Remark}

\begin{document}

\title[Best constants in inequalities]
{Best constants in inequalities involving analytic and co-analytic projections and Riesz's theorem in various function spaces}

\author{Petar  Melentijevi\'{c}}
\address{
Faculty of  Mathematics\endgraf
University of Belgrade\endgraf
Studentski trg 16\endgraf
11000 Beograd\endgraf
Serbia\endgraf}
\email{petarmel@matf.bg.ac.rs}

\author{Marijan Markovi\'{c}}
\address{
Faculty of  Sciences and Mathematics\endgraf
University of Montenegro\endgraf
D\v{z}ord\v{z}a Va\v{s}ingtona bb\endgraf
81000 Podgorica\endgraf
Montenegro\endgraf}
\email{marijanmmarkovic@gmail.com}

\begin{abstract}
Let $P_+$ be the Riesz's projection operator and let  $P_-= I - P_+$. We consider estimates of the expression $\|( |P_ +  f | ^s + |P_-  f |^s)
^{\frac 1s}\|_{L^p (\mathbf{T})}$ in terms of Lebesgue $p$-norm of  the function  $f \in L^p(\mathbf{T})$. We find the accurate estimates for
$p\ge 2$ and $0<s\le p$, thus significantly improving results from \cite{KALAJ.TAMS} where it is  considered for $s=2$ and  $1<p<\infty$. Interestingly, for this range of $s$ there holds the appropriate vector-valued inequality with the same constant.  Also,
we obtain the right asymptotic of the constants for large $s$.  This proves  the conjecture
of Hollenbeck and Verbitsky on the Riesz projection operator in some cases. As a consequence of inequalities  we have proved in the paper we get
Riesz-type theorems on conjugate  harmonic functions  for various function spaces.  In particular, slightly general version of Stout's theorem
for Lumer Hardy spaces is  obtained by a new approach.
\end{abstract}



\maketitle

\section{Introduction}


Let $\mathbf{U}=\{z\in \mathbf{C}:|z|<1\}$ be the unit disk  and let $\mathbf{T}=\{z\in \mathbf{C}:|z|=1\}$  be the  unit circle. For  $0<p\le\infty$
we denote by $L^p(\mathbf{T})$ the  Lebesgue space  on $\mathbf{T}$. The space $H^p(\mathbf{T})$ contains  all $\varphi  \in  L^p(\mathbf{T})$   for
which  all negative  Fourier   coefficients are equal to zero, i.e.,
\begin{equation*}
\hat{\varphi}(n)  =  \frac{1}{2\pi}\int_{0}^{2\pi} \varphi (e^{it})e^{-int} dt = 0\quad \text{for every  integer}\  n<0.
\end{equation*}
The Riesz projection  operator $P_+:L^p(\mathbf{T})\to H^p(\mathbf{T})$,  and the  co-analytic projection operator  $P_- = I- P_+$,  where $I$ is the
identity operator on  $L^p(\mathbf{T})$,   operate in the following  way
\begin{equation*}
P_+f (\zeta)  =  \sum_{n=0}^{+\infty}      \hat{f}(n)  \zeta^n,\quad   P_- f(\zeta)=\sum_{n= - \infty}^{-1} \hat{f}(n) \zeta^n,
\quad f(\zeta)= \sum_{n =   -\infty}^{+\infty} \hat{f}(n) \zeta^n.
\end{equation*}

The Riesz projection operator is not bounded on $L^1(\mathbf{T})$; actually  it does not exist  a bounded projection of $L^1(\mathbf{T})$ onto
$H^1(\mathbf{T})$. This was shown by Newman, and  generalized by  Rudin.  From the Parseval identity  it  follows that the  operator $P_+$ is an
orthogonal  projection  of  $L^2(\mathbf{T})$ onto $H^2(\mathbf{T})$. Therefore,  the norm of $P_+$ as an operator from  ${L^2(\mathbf{T})}$
onto   $H^2(\mathbf{T})$ is equal to  $1$.  If   $p=\infty$, then the  Riesz projection operator  $P_+$  is also  unbounded
The question for $p\ne 2$ is more subtle.   It is a classical result that Riesz  projection $P_+$ is a bounded operator considered as   an operator
on $L^p(\mathbf{T})$  for every $1<p<\infty$. This  is  the classical result  obtained  by M. Riesz. Nowadays  there exist  many  different  proofs
of this fact.  There exists an elementary  proof of Stein  base    on Hardy--Stein equality.  But it  needed a long time to  find the exact norm of
the Riesz projection. This is solved by Hollenbeck and Verbitsky in 2000.

For a function  $f$ on $\mathbf{U}$ and $r\in (0,1)$ we denote by $f_r$ the function  $f_r (\zeta) =  f(r\zeta)$, $\zeta\in\overline{\mathbf{U}}$. The
harmonic Hardy space $h^p$, $0<p\le \infty$, consists of all  harmonic complex-valued  functions  $U$  on $\mathbf{U}$ for  which the integral mean
\begin{equation*}
M_p(U,r)  = \left\{ \int_\mathbf{T} |U_r(\zeta)|^p\frac{|d\zeta|}{2\pi}\right\}^{\frac 1p}
\end{equation*}
remains bounded as $r$ approaches $1$.  Since $|U|^p$, $1\le p<\infty$ is  subharmonic on  $\mathbf{U}$,  the integral mean $M_p(U,r)$ is increasing in
$r$. The  norm  on $h^p$ in this case  is  given by
\begin{equation*}
\|U\|_p = \lim_{r\rightarrow 1} M_p(U,r).
\end{equation*}
The analytic Hardy space $H^p$ is the subspace of $h^p$ that contains analytic functions. For the  theory of  Hardy spaces in the unit disk we refer to
classical books \cite{DUREN.BOOK, FISHER.BOOK, GARNETT.BOOK, RUDIN.BOOK} and  to the more recent  Pavlovi\'{c} book \cite{PAVLOVIC.BOOK} for new results
on the Hardy  space theory in the  unit disc. Since $|f|^p$ is subharmonic  for every $0<p<\infty$,  the norm on the   Hardy space $H^p$ may  be
introduces as
\begin{equation*}
\|f\|_p  = \lim_{ r \rightarrow 1}  M(r,f).
\end{equation*}
For $p=\infty$ the space $H^\infty$ contains all bounded analytic functions in  the unit disc.

It  is well   known that   for $f\in H^p$ the  radial  boundary value   $f^{\ast}(\zeta)= \lim_{r\rightarrow 1-} f(r\zeta)$ exists for almost every
$\zeta\in \mathbf{T}$,  $f^\ast$ belongs to the space  $H^p(\mathbf{T})$, and   we have the isometry relation $\|f^{\ast}\|_{L^p(\mathbf{T})} =
\|f\|_{H^p}$. Moreover, given $\varphi \in H^p (\mathbf{T})$, the Cauchy (and the  Poisson)  extension  of it  gives  a function  $f$ in  $H^p$ such
that   $f^\ast (\zeta ) = \varphi(\zeta)$ for almost every $\zeta\in \mathbf{T}$. Therefore, one may identify the spaces $H^p(\mathbf{T})$  and  $H^p$
via the   isometry $f\rightarrow f^\ast$.  If we identify  the Hardy space $H^p$ and  the  space  $H^p(\mathbf{T})$, then the Riesz operator   $P_+$
may be  represented as the  Cauchy    integral  (with density $f^\ast$)
\begin{equation*}
P_+ f (z) =  \frac{1}{2\pi i} \int_{\mathbf{T}} \frac{f^\ast(\zeta)}{\zeta - z} d\zeta, \quad z \in \mathbf{U}.
\end{equation*}


This paper is mainly motivated by this   conjecture of Hollenbeck  and  Verbitsky posed form  2010. Partial  results are known. Kalaj in 2017 confirmed
the conjecture in the case $s=2$.  The case $s=\infty$ is considered  in the paper \cite{HV.JFA} where the authors invented the method based on subharmonic
minorants. This metod  is   also applied quite recently  in \cite{KALAJ.TAMS} in order to prove the case $s=2$. 

Hollenbeck  and  Verbitsky \cite{HV.OTAA} posed  a  problem to   find  the optimal constant $C _{s,p }$   for the inequality
\begin{equation*}
\|  ( |P_ +  f | ^s  +  |P_-  f |^s) ^{\frac 1s}  \|_{L^p (\mathbf{T})}\le C_{s,p } \|f\|_{L^p (\mathbf{T})},
\quad f\in {L^p (\mathbf{T})}, \,    1<p<\infty, \, 0<s<\infty.
\end{equation*}
They   conjectured   that
\begin{equation*}
C _ {s,p } = \frac { 2^{\frac 1 s} }{ 2\cos \frac \pi {2p}  },  \quad   1<p<2,\  0<s<\sec^2\frac \pi  {2p},
\end{equation*}
and
\begin{equation*}
C _{s,p } = \frac { 2^{\frac 1s} }{ 2\sin \frac \pi {2p} },  \quad  2<p<\infty, \ 0<s<\csc^2\frac \pi {2p }.
\end{equation*}

The  Hollenbeck  and  Verbitsky result from \cite{HV.JFA} is that $C_{\infty, p }  = \frac  1 {\sin  \frac  \pi p}$,   $1<p<\infty$. In other words,
this result says that
\begin{equation*}
\|\max \{ |P_ + (f)|,  |P_- (f)|\} \|_{L^p  (\mathbf{T})}\le \frac  1 {\sin  \frac  \pi p} \|f\|_{L^p (\mathbf{T})},  \quad  f\in {L^p (\mathbf{T})}.
\end{equation*}
This   case  is  solved in 2000 by Hollenbeck  and  Verbitsky  and  the problem    has a  long history.
Some partial results on the norm of the Riesz  operator $P_+$ were known much earlier, obtained by  Gohberg and Krupnik. What Hollenbeck and Verbitsky
proved was known  as Gohberg   and    Krupnik  conjecture  (Gohberg   and    Krupnik   proved   some special cases of the conjecture, i.e.,  for the
numbers which may be  represented in the  form   $p=2^n$,  $n\in\mathbf{N}$).


Now  we will state the main results of the paper.

\begin{theorem}\label{TH.1}
For $1<p\le  \frac{5}{4}$  and  $0<s\le 4$ or $\frac{5}{4}<p<2$ and $0<s\le 2$ we have
\begin{equation*}
{C}_{s,p}=\frac{2^{ \frac {1}{s}}}{2\cos\frac{\pi}{2p}}.
\end{equation*}

For $p\ge 2 $ and $0<s\le p$   we have
\begin{equation*}
{C}_{s,p}=\frac{2^{ \frac {1}{s}}}{2\sin\frac{\pi}{2p}}.
\end{equation*}
\end{theorem}

This    theorem contains the Kalaj's result (the case  $s=2$ in the above theorem).

We confirm  the asymptotic conjectured by Hollenbeck and Verbitsky \cite{HV.OTAA}.  This is the content of our second result given in the next theorem.

\begin{theorem}\label{TH.2}
We have
\begin{equation*}
\lim_{s \rightarrow \infty}C_{s,p}= \frac{1}{\sin\frac{\pi}{p}}
\end{equation*}
for every $1<p<\infty$.
\end{theorem}

We will   show that  the constant we have in some inequalities and which are optimal for  functions in   Hardy space   are useful in some other
function  spaces such as Lumer's Hardy spaces,    weighted Bergman spaces and Fock    spaces. Finally, we  discuss the  vector--valued case. We
would like to cover the vector--valued case, i.e., we may  suppose  that $f$  is  vector-valued,
i.e., it may  maps  the unit  ball into  $\mathbf{C}^n$.  It is maybe  surprising that  the constants    $C_{s,p} $   remains the same.
At  the end  of the     paper  we considered the   vector-valued analytic  mappings,  where we can also use  the method of plurisubharmonic
minorants.

Let us mention that by a result of Marczinkiewicz and Zygmund we have the following vector inequalities with the same constants as above:

\begin{corollary}
For $2\leq s \leq p$ and $1<p\leq s\leq 2$	there holds the estimates

$$  \|\big(\sum_{k=1}^{\infty}|P_{\pm}f_k|^s\big)^{\frac{1}{s}}\| _{_{L^p({\mathbb{T}})}} \leq C_{s,p}   \|\bigg(\sum_{k=1}^{\infty}|f_k|^s\bigg)^{\frac{1}{s}}\|  _{L^p({\mathbb{T}})}.  $$                                                                      	
\end{corollary}	
The Riesz projection is closely related  to the Hilbert operator   $H$  which  maps a function  $u$ in harmonic Hardy space $h^p(\mathbf{U})$ to
the  harmonic  conjugate  $\tilde{u}$,   i.e., a function  such that  $u+i\tilde{u}$ is an analytic  function on the  unit disc. The classical result
of Pichorides  \cite{PICHORIDES.STUDIA} gives $\|H\| =\tan\frac{\pi}{2\tilde{p}}$, where $\tilde{p}= \min\{p,\frac{p}{p-1}\}$. It is not hard to derive
this result from  the  result of Hollenbeck and Verbitsky.

\section{Lower estimates of the constants $C_{s,p}$}

In this section we  estimate  from below for constants $C_{s,p}$   for $0<s<\infty$ and $1<p<\infty$.
We will use the family of test functions defined by
\begin{equation*}
f_{\gamma}= \alpha \mathrm {Re}\,  g_{\gamma} + i \beta \mathrm {Im}\, g_{\gamma},
\end{equation*}
where $g_{\gamma}(z)= \left(\frac{1+z}{1-z}\right)^{\frac{2\gamma}{\pi}}$.
Note that  $|\mathrm {Re}\,  g_{\gamma}|=  \tan \gamma |\mathrm {Im}\,   g_{\gamma}|$.  Hence,  if  we let  $\gamma\to  \frac{\pi}{p}$,  we obtain
\begin{equation*}
\frac{\|\left(|P_+f|^s+|P_-f|^s \right)^{\frac{1}{s}}\|_{L^p(\mathbf{T})}}{\|f\|_{L^p(\mathbf{T})}}=
\frac{\left(|\alpha+\beta|^s+|\alpha-\beta|^s\right)^{\frac{1}{s}}}{2\left(\alpha^2 \cos^2 \frac{\pi}{2p}
	+ \beta^2 \sin^2 \frac{\pi}{2p}\right)^{\frac{1}{2}}}.
\end{equation*}
Therefore, we easily conclude that
\begin{equation*}
C_{s,p} \ge \sup_{\alpha,\beta \in \mathbb{R}}
\frac{\left(|\alpha+\beta|^s+|\alpha-\beta|^s\right)^{\frac{1}{s}}}{2\left(\alpha^2 \cos^2 \frac{\pi}{2p}
	+ \beta^2\sin^2 \frac{\pi}{2p}\right)^{\frac{1}{2}}}.
\end{equation*}
From the preceding  inequality   we conclude    that it is enough to consider  the function
\begin{equation*}
T(\alpha,\beta)= \frac{\left(|\alpha+\beta|^s+|\alpha-\beta|^s\right)^{\frac{1}{s}}}
{2\left(\alpha^2 \cos^2 \frac{\pi}{2p} + \beta^2 \sin^2 \frac{\pi}{2p}\right)^{\frac{1}{2}}}.
\end{equation*}
This function is even in both variables, i.e., $T(\alpha, \beta) = T(-\alpha, -\beta)$ and  without loss of generality,  we can moreover
assume  that $|\alpha+\beta| \ge  |\alpha-\beta|$. Also, $T(\alpha,\beta)$ it is homogeneous,
so   we can set $\alpha+\beta=1$ and $\alpha-\beta=t$,   $|t|\le 1$.  Therefore, we get
\begin{equation*}\begin{split}
\frac{\left(|\alpha+\beta|^s+|\alpha-\beta|^s\right)^{\frac{1}{s}}}{2\left(\alpha^2 \cos^2 \frac{\pi}{2p} + \beta^2 \sin^2
	\frac{\pi}{2p}\right)^{\frac{1}{2}}}&
=\frac{(1+|t|^s)^{\frac{1}{s}}}{2 ((\frac{1+t}{2})^2 \cos^2 \frac{\pi}{2p} + (\frac{1-t}{2})^2 \sin^2 \frac{\pi}{2p})^{\frac{1}{2}}}
\\&=\frac{(1+|t|^s)^{\frac{1}{s}}}{\sqrt{1+t^2+2t\cos \frac{\pi}{p}}}.
\end{split}\end{equation*}
Hence, the range of the function $T(\alpha,\beta)$ is the union of the values
\begin{equation*}
\frac{(1+t^s)^{\frac{1}{s}}}{\sqrt{1+t^2+2t\cos \frac{\pi}{p}}}
\quad \text{and}\quad \frac{(1+t^s)^{\frac{1}{s}}}{\sqrt{1+t^2-2t\cos \frac{\pi}{p}}}
\end{equation*}
for $0 \le   t \le  1$.

Since, for $1<p\le 2$ we have
\begin{equation*}
\sqrt{1+t^2-2t\cos \frac{\pi}{p}} \ge  \sqrt{1+t^2+2t\cos \frac{\pi}{p}},
\end{equation*}
we may  conclude
\begin{equation*}
C_{s,p} \ge \sup_{t \in [0,1]} \frac{(1+t^s)^{\frac{1}{s}}}{\sqrt{1+t^2+2t\cos \frac{\pi}{p}}}.
\end{equation*}

Let us now consider   $F(t)=\frac{(1+t^s)^2}{(1+t^2+2t\cos\frac{\pi}{p})^s}$. We  find
$F'(t)=\frac{2s(1+t^s)}{(1+t^2+2t\cos\frac{\pi}{p})^{s+1}}[t^{s-1}-t+\cos\frac{\pi}{p}(t^s-1)]$,
and, therefore $ \mathrm {sgn}\,  F'(t) =  \mathrm {sgn}\,  \Phi(t)$,  where $\Phi(t)= t^{s-1}-t+\cos\frac{\pi}{p}(t^s-1)$.

We have  $\Phi'(t)= (s-1)t^{s-2}-1+st^{s-1}\cos\frac{\pi}{p}$ and $\Phi''(t)=(s-1)t^{s-3}(ts\cos\frac{\pi}{p}+s-2)$,  and we differ the following
two cases:

1) If $s\cos\frac{\pi}{p}+s-2 \geq 0,$ then $\Phi''(t)>0,$ $\Phi'$ is increasing and takes the values between $\Phi'(0)=-1$ and $\Phi'(1)= s-2+s\cos\frac{\pi}{p} \geq 0,$ so there is an $t_0 \in (0,1)$ such that $\Phi'(t_0)=0.$ Thus $\Phi$ decreases on $(0,t_0)$ and increases on
$(t_0,1),$ has the minimum in $t_0$ smaller than zero $(\Phi(1)=0),$ and consequently $\Phi$ is positive on $(0,\tilde{t})$ and negative on
$(\tilde{t},1).$ So, $F$ is increasing from $0$ to $\tilde{t},$ decreasing from  $\tilde{t}$ to $1,$ therefore $\max_{t \in [0,1]} F(t)=
F(\tilde{t}),$ where $\tilde{t}$ is a unique solution of the equation: $t^{s-1}-t+\cos\frac{\pi}{p}(t^s-1)=0,$ for $s>2$ and $t \in [0,1).$

2)  If $s\cos\frac{\pi}{p}+s-2 < 0,$ then for $t_1=\frac{2-s}{s\cos\frac{\pi}{p}}$ $\quad$ $\Phi''$ has the zero. Consequently $\Phi'$
increases on $[0,t_1]$ and decreases on $[t_1,1],$ so $\Phi'<0$ and $\Phi$ decreases, which gives us $\Phi(t) \geq \Phi(1)=0.$ Therefore, in
this case, $F$ increases and $\max_{t \in [0,1]} F(t)= F(1)= \frac{4^{1-s}}{\cos^{2s}\frac{\pi}{2p}}.$

The rest easily follows from the simple observation: If $p>2,$ than $\cos\frac{\pi}{p}=-\cos\frac{\pi}{p'}$ for dual exponent $p'=\frac{p}{p-1}<2.$ Consequently, we can change the appropriate signs and functions $\sin$ and $\cos$.

\begin{remark}
For the other values of $s>0$ we have the following estimates
\begin{equation*}
{C}_{s,p} \ge  \frac{2^{\frac {1}{s}}}{2\cos\frac{\pi}{2p}},\quad  2\le  s \le  \frac{1}{\cos^2 \frac{\pi}{2p}},
\end{equation*}
and
\begin{equation*}	
{C}_{s,p} \ge  \frac{(1+\tilde{t}^s)^{\frac{1}{s}}}{\sqrt{1+\tilde{t}^2+2\tilde{t}\cos\frac{\pi}{p}}},\quad s\ge  \frac{1}{\cos^2 \frac{\pi}{2p}},
\end{equation*}
where  $\tilde{t}$ is the unique zero in $(0,1)$ of the equation $t^{s-1}-t+\cos\frac{\pi}{p}(t^s-1)=0$.

Also, we have the analogous estimates from the above with the appropriate change of $\cos$ with $\sin-$function in both cases and the expressions for the constant.
\end{remark}

\section{Asymptotic of the constants $C_{s,p}$ as  $s\to \infty$ and the proof of Theorem \ref{TH.2}}
From our considerations from the previous subsection we conclude that:
\begin{equation*}
C_{s,p} \ge \frac{\big(1+\tilde{t}^s\big)^{\frac{1}{s}}}{\sqrt{1+\tilde{t}^2+2\tilde{t}\cos \frac{\pi}{p}}},
\end{equation*}
where $\tilde{t} \in (0,1)$ is a unique zero of the equation $t^{s-1}-t+\cos\frac{\pi}{p}(t^s-1)=0,$ for $s$ big
enough and $1<p<2$.

Also, analyzing the case $s\cos\frac{\pi}{p}+s-2 \geq 0,$ we see that $\Phi'(-\cos\frac{\pi}{p})=(s-1)|\cos\frac{\pi}{p}|^{s-2}-1+s|\cos\frac{\pi}{p}|^{s-1}\cos\frac{\pi}{p}=
|\cos\frac{\pi}{p}|^{s-2}(s-1-s\cos^2\frac{\pi}{p})>0,$ since $s-1-s\cos^2\frac{\pi}{p}=s-2-s\cos^2
\frac{\pi}{p}+1-s(\cos^2\frac{\pi}{p}+\cos\frac{\pi}{p})\geq 1,$ because of $s\cos\frac{\pi}{p}+s-2 \ge 0$
and $\cos\frac{\pi}{p}>0$. Therefore, $\Phi$ is positive on $(0,\tilde{t})$ and negative on $(\tilde{t},-\cos\frac{\pi}{p})$,
consequently $\tilde{t}<|\cos\frac{\pi}{p}|$. Hence, $\tilde{t}=-\cos\frac{\pi}{p}+\tilde{t}^{s-1}(1+\tilde{t}\cos\frac{\pi}{p})$
and  $\tilde{t}\rightarrow -\cos\frac{\pi}{p},$ as $s \rightarrow \infty,$ which gives:
\begin{equation*}
\liminf_{s \rightarrow \infty} C_{s,p} \geq \frac{1}{\sin\frac{\pi}{p}}.
\end{equation*}

To obtain the inequality for the same limit from the above, we evoke the following estimate obtained  in  \cite{HV.JFA}
\begin{equation*}
\max\{|z|^p,|w|^p\} \le   \frac{1}{\sin^p\frac{\pi}{p}}|z+\overline{w}|^p-b_p \mathrm {Re}\, (zw)^{\frac{p}{2}}.
\end{equation*}
We choose an arbitrary small $\epsilon>0$ and multiplies both sides by $(1+\epsilon)^p$,  we get
\begin{equation*}
(1+\epsilon)^p \max\{|z|^p,|w|^p\} \leq \frac{(1+\epsilon)^p}{\sin^p\frac{\pi}{p}}|z+\overline{w}|^p-b_p(1+\epsilon)^p \mathrm {Re}\, (zw)^{\frac{p}{2}}.
\end{equation*}

Choosing $s$ such that $(1+\epsilon)^s \ge  2,$ we have
\begin{equation*}
(|z|^s+|w|^s)^{\frac{p}{s}} \leq 2^{\frac{p}{s}} \max\{|z|^p,|w|^p\} \le  (1+\epsilon)^p \max\{|z|^p,|w|^p\}.
\end{equation*}
Hence, we obtain
\begin{equation*}
(|z|^s+|w|^s)^{\frac{p}{s}} \le  \frac{(1+\epsilon)^p}{\sin^p\frac{\pi}{p}}|z+\overline{w}|^p-b_p(1+\epsilon)^p \mathrm {Re}\,  (zw)^{\frac{p}{2}},
\end{equation*}
from which, after integrating with $z=P_+f, w=P_-f$ we get
\begin{equation*}
\limsup_{s \rightarrow \infty}C_{s,p} \le  \frac{1+\epsilon}{\sin\frac{\pi}{p}}.
\end{equation*}
Hence, we obtain
\begin{equation*}
\lim_{s \rightarrow \infty}C_{s,p}= \frac{1}{\sin\frac{\pi}{p}},
\end{equation*}
as conjectured in \cite{HV.OTAA}.

Proof in the case $p>2$ is quite similar.


\section{Proof of Theorem \ref{TH.1} }
Here we prove the result given in Theorem   \ref{TH.1}. This section is devoted to the  proof of Theorem  \ref{TH.1}.

Let us, first, note that if it is proved for some value of $s$  that then it will follow  for all smaller values of $s>0$. Namely, let us suppose that
we have proved our theorem for some $s=s_0$.  Then, if  we use inequalities between  means  of order   $s$  and of  order  $s_0$,  for  $0<s\le  s_0$
we obtain
\begin{equation*}
\left(\frac{|P_+f(\zeta)|^s+|P_-f(\zeta)|^s}{2}\right)^{\frac {p}{s}}
\le \left(\frac{|P_+f(\zeta)|^{s_0}+|P_-f(\zeta)|^{s_0}}{2}\right)^{\frac{p}{s_0}},
\end{equation*}
i.e.,
\begin{equation*}
\left(|P_+f(\zeta)|^s+|P_-f(\zeta)|^s\right)^{\frac{p}{s}} \le 2^{\frac{p}{s}- \frac{p}{s_0}}
\left(|P_+f(\zeta)|^{s_0}+|P_-f(\zeta)|^{s_0}\right)^{\frac{p}{s_0}}.
\end{equation*}
If we integrate the last inequality, we obtain
\begin{equation*}
\int_{\mathbf{T}}\left(|P_+f(\zeta)|^s+|P_-f(\zeta)|^s\right)^{\frac{p}{s}} |d\zeta|
\le 2^{\frac{p}{s}-\frac{p}{s_0}} \int_{\mathbf{T}}\left(|P_+f(\zeta)|^{s_0}+|P_-f(\zeta)|^{s_0}\right)^{\frac  p{s_0}} |d\zeta| .
\end{equation*}
Using the result for $s=s_0$, i.e., the inequality
\begin{equation*}
\int_{\mathbf{T}} \left(|P_+f(\zeta)|^{s_0}+|P_-f(\zeta)|^{s_0}\right)^{\frac p{s_0}} |d\zeta| \le
\frac{2^{\frac p{s_0}}} {2^p\cos^p \frac{\pi}{2p}} \int_{\mathbf{T}} |f(\zeta)|^p |d\zeta|,
\end{equation*}
we obtain
\begin{equation*}
\left(\int_{\mathbf{T}} \left(|P_+f(\zeta)|^s+|P_-f(\zeta)|^s\right)^{\frac ps} |d\zeta|\right)^{\frac 1p}
\le \frac{2^{\frac 1s}}{2\cos \frac{\pi}{2p}}\left\{ \int_{\mathbf{T}} |f(\zeta)|^p |d\zeta|\right\}^{\frac 1p}.
\end{equation*}
	
We can act similarly    in the case $2<p<\infty$, assuming that we have
\begin{equation*}
\left(\int_{\mathbf{T}}\left(|P_+f(\zeta)|^{s_0}+|P_-f(\zeta)|^{s_0}\right)^{\frac {p}{s_0}}|d\zeta|\right)^{\frac{1}{p}}\le \frac{2^{\frac{1}{s_0}}}{2\sin\frac{\pi}{2p}}\left(\int_{\mathbf{T}}|f(\zeta)|^p |d\zeta|\right)^{\frac{1}{p}}.
\end{equation*}
If we use again  the  inequality between means  of order   $s$ and   $s_0$ we obtain
\begin{equation*}
\int_{\mathbf{T}}\left(|P_+f(\zeta)|^s+|P_-f(\zeta)|^s\right)^{\frac ps}|d\zeta|\le 2^{\frac ps- \frac p{s_0}}
\int_{\mathbf{T}}\left(|P_+f(\zeta)|^{s_0}+|P_-f(\zeta)|^{s_0}\right)^{\frac p{s_0}}|d\zeta|,
\end{equation*}
i.e.,   the appropriate inequality for any $s \le  s_0$.

The proof of this  upper bound  for  the  $s$-norm of the Riesz projection $P_+$  and co-analytic projection $P_-$ depends on some nontrivial sharp
inequalities. We will use  some subharmonic minorants as in \cite{KALAJ.TAMS}. Therefore, our  upper estimate is a consequence of a result for
complex numbers. We use some new observations and hence, our inequality is stronger than that in \cite{KALAJ.TAMS}.

\subsection{The case  $1<p<2$}  The  following lemma is crucial.

\begin{lemma}
For   $s=4$ and  $1<p<\frac{5}{4}$   there  holds
\begin{equation*}
-\left(\frac{|z|^s+|w|^s}{2}\right)^{\frac ps}+\frac{|z+\overline{w}|^p}{2^p\cos^p \frac{\pi}{2p} } -  \tan\frac{\pi}{2p}\mathrm {Re} (zw)^{\frac  p2}
\ge 0,\quad z\in \mathbf{C},\, w\in \mathbf{C}.
\end{equation*}
If $s=2$, then this inequality is valid   in  the range $\frac{5}{4}<p<2$.
\end{lemma}

\begin{proof}
Since of  homogeneity, if we divide it by  $\max\{|z|^p,|w|^p\}$  and if we write  the number ${\overline{zw}}/{|z|^2}$ in the  polar  form, we obtain
that   we have to prove only the following:  for    $0\le  r\le 1$ and  $-\pi\le  t \le\pi$  there holds
\begin{equation*}
-\left(\frac{1+r^s}{2}\right)^{\frac{p}{s}}+\frac{\left(1+r^2+2r\cos t\right)^{\frac{p}{2}}}{2^p
\cos^p\frac {\pi}{2p}} -r^{\frac{p}{2}}\tan\frac{\pi}{2p}\cos\frac{tp}{2} \ge 0.
\end{equation*}
Let us denote the left side by  $\Phi(r,t)$. For fixed   $0<s\le 2$, we will found critical points of this function in the interior  of $(r,t) \in
[0,1]\times  [0,\pi]$ (since  $\Phi$ is even  in    $t$)  as  zeroes   of the partial  derivatives  with respect to   $r$  and  $t$. Note that,
since of  the following inequality
\begin{equation*}\begin{split}
&-\left(\frac{1+r^s}{2}\right)^{\frac{p}{s}}+\frac{\left(1+r^2+2r\cos t\right)^{\frac{p}{2}}}{2^p
	\cos^p\frac {\pi}{2p}} -r^{\frac{p}{2}}\tan\frac{\pi}{2p}\cos\frac{tp}{2}
 \ge
\\&-\left(\frac{1+r^s}{2}\right)^{\frac{p}{s}}+\frac{\left(1+r^2+2r\cos \frac{\pi}{p} \right)^{\frac{p}{2}}}{2^p
	\cos^p\frac {\pi}{2p}}
\end{split}\end{equation*}
for $t \in [\frac{\pi}{p},\pi]$. Since the last expression is non-negative by considerations in the second section, it is enough to prove the
inequality for $t \in [0,\frac{\pi}{p}].$

The potential stationary points satisfy the following equations $ \frac{2}{p}\frac{\partial \Phi}{\partial r}(r,t)=0$ which is
equivalent with
\begin{equation*}
-r^{s-1}\left(\frac{1+r^s}{2}\right)^{\frac{p}{s}-1}
+\frac{2(r+\cos t)} {2^p\cos^p\frac{\pi}{2p}} (1+r^2+2r\cos t )^{\frac p2-1}
-r^{\frac{p}{2}-1}\tan\frac{\pi}{2p}\cos\frac{tp}{2}=0,
\end{equation*}
and  $ \frac{2}{p}\frac{\partial \Phi}{\partial t}(r,t)=0$ is equivalent to the following one
\begin{equation*}
\frac{2}{p}\frac{\partial \Phi}{\partial t}(r,t)=\frac{-2r\sin t}{2^p\cos^p\frac{\pi}{2p}}
 (1+r^2+2r\cos t)^{\frac{p}{2}-1}+r^{\frac{p}{2}}
\tan\frac{\pi}{2p}\sin\frac{tp}{2}=0.
\end{equation*}
From the second equation we have
\begin{equation*}
\frac{ (1+r^2+2r\cos t  )^{\frac p2-1}}{2^pr^{\frac p2}\cos^p\frac{\pi}{2p}\tan\frac{\pi}{2p}}=
\frac{\sin\frac{tp}{2}}{2r\sin t}.
\end{equation*}
Now, using this equality, we transform the expression for $\frac{2}{p}\frac{\partial \Phi}{\partial r}(r,t)=0$ in the following way
\begin{equation*}\begin{split}
&-r^{s}\left(\frac{1+r^s}{2}\right)^{\frac{p}{s}-1}+\frac{(r+\cos t)r^{\frac p2}
\tan\frac{\pi}{2p}\sin\frac{tp}2}{\sin t}-r^{\frac p2}\tan\frac{\pi}{2p}\cos\frac{tp}{2}
\\&=-r^{s}\left(\frac{1+r^s}{2}\right)^{\frac ps-1}+\frac{r^{\frac p2} \tan\frac{\pi}{2p}}{\sin t}
 \left((r+\cos t)\sin\frac{tp}{2}-\sin t\cos\frac{tp}{2}\right)
\\&=-r^{s}\left(\frac{1+r^s}{2}\right)^{\frac{p}{s}-1}+
\frac{r^{\frac p2}\tan\frac{\pi}{2p}}{\sin t}\bigg(r\sin\frac{tp}{2}-\sin\big(t-\frac{tp}{2}\big)\bigg)=0.
\end{split}\end{equation*}
In the   critical   points   the   function  $\Phi $  is equal to the following expressions
\begin{equation*}\begin{split}
\Phi(r,t)
& = -\frac{r^{\frac{p}{2}}\tan{\frac{\pi}{2p}}}{2\sin t}
\left(r\sin\frac{tp}{2}-\sin\big(t-\frac{tp}{2}\big)\right)(1+r^{-s}) \\&+r^{\frac{p}{2}}\tan\frac{\pi}{2p}\frac{(1+r^2+2r\cos t)\sin\frac{tp}{2}}{2r\sin t} -r^{\frac p2} \tan\frac{\pi}{2p}\cos\frac{tp}{2}
\\&=\frac{r^{\frac{p}{2}}\tan\frac{\pi}{2p}}{2r\sin t} \bigg[-(r+r^{1-s})\bigg(r\sin\frac{tp}{2}-\sin\big(t-\frac{tp}{2}\big)\bigg)
\\&+(1+r^2+2r\cos t) \sin\frac{tp}{2}-2r\cos\frac{tp}{2}\sin t \bigg]
\\&=\frac{r^{\frac p2} \tan\frac{\pi}{2p}}{2r\sin t}
\bigg[-(r+r^{1-s})\bigg(r\sin\frac{tp}{2}-\sin\big(t-\frac{tp}{2}\big)\bigg)
\\&+(1+r^2)\sin\frac{tp}{2}-2r\sin\big(t-\frac{tp}{2}\big)\bigg].
\end{split}\end{equation*}

In    the  critical    points  we have  $\Phi(r,t)\ge   0$  only if there holds the inequality
\begin{equation*}
 \left(r^{1-s}-r\right)\sin\big(t-\frac{tp}{2}\big)+\left(1-r^{2-s}\right)\sin\frac{tp}{2}\ge 0.
\end{equation*}
If we multiply it   by the factor  $r^s$,    we obtain an equivalent one
\begin{equation*}
 (r-r^{s+1} )\sin \left(t-\frac{tp}{2}\right)+\sin\frac{tp}{2}\left(r^s-r^2\right)\ge 0 .
\end{equation*}
But this inequality is  obviously  true  if $0<s<2$, since  we have
\begin{equation*}
r-r^{s+1}\ge 0,\quad  \sin\left(t-\frac{tp}{2}\right)\ge 0,\quad   \sin\frac{tp}{2}\ge  0,\quad   r^s-r^2\ge 0
\end{equation*}
if   the numbers    $p$, $r$  and  $t$  belong to the    intervals.

Now, we have to check the inequality at the boundary point of the rectangle $(r,t) \in [0,1]\times  [0,\frac{\pi}{p}]$.
(This is so, because we prove the main inequality for $t \in [\frac{\pi}{p},\pi]$.) Also, we easily see that the case of
$t=\frac{\pi}{p}$ follows from the results of the second section.

For  $t=0,$  the  inequality we want  to prove is
\begin{equation*}
-\left(\frac{1+r^s}{2}\right)^{\frac{p}{s}}+ \frac{(1+r)^p}{2^p\cos^p\frac{\pi}{2p}}-r^{\frac{p}{2}}\tan\frac{\pi}{2p} \ge  0.
\end{equation*}
Because of
\begin{equation*}\begin{split}
& -\left(\frac{1+r^s}{2}\right )^{\frac ps}+ \frac{(1+r)^p}{2^p\cos^p\frac{\pi}{2p}}-r^{\frac p2}\tan\frac{\pi}{2p} \ge
 -\left(\frac{1+r^2}{2}\right)^{\frac p2}+ \frac{(1+r)^p}{2^p\cos^p\frac{\pi}{2p}}-r^{\frac p2}\tan\frac{\pi}{2p}=\tilde{F}(r),
\end{split} \end{equation*}
it is enough  to show that   $\tilde{F}(r) \ge 0$,  which is equivalent to
\begin{equation*}\begin{split}
F(r)=\frac{\tilde{F}(r)}{(1+r)^p} = - \left(\frac{1+r^2}{2(1+r)^2}\right)
^{\frac {p}{2}}+\frac{1}{2^p\cos^p\frac{\pi}{2p}}-\left(\frac{r}{(1+r)^2}\right)^{\frac{p} {2}}
\tan\frac{\pi}{2p} \ge0.
\end{split}\end{equation*}

We will  prove  that  $F$  is decreasing. From   $F(r)\ge F(1)$,  i.e.,  from  $F(1) \ge  0$ (this may be considered as  $r=1$)
follows the desired conclusion. Indeed,
\begin{equation*}
F'(r)=\frac{p(r-1)}{2(1+r)^3} \left(-\left(\frac{1+r^2}{2(1+r)^2}\right)
^{\frac{p}{2}-1}+\left(\frac{r}{(1+r)^2}\right)^{\frac{p}{2}-1}\tan\frac{\pi}{2p}\right) \le 0,
\end{equation*}
since
\begin{equation*}
-\left(\frac{1+r^2}{2(1+r)^2}\right)^{\frac{p}{2}-1}+\left(\frac{r}{(1+r)^2}\right)^{\frac{p}{2}-1}\tan\frac{\pi}{2p} \ge  0
\end{equation*}
is equivalent to
\begin{equation*}
\left(\frac{1+r^2}{2r}\right)^{\frac{p}{2}-1} \le  \tan\frac{\pi}{2p},
\end{equation*}
which is correct,  since for  considered  values of   $r$ and  $p$ holds  $\tan\frac{\pi}{2p}\ge  1$ and $\frac{1+r^2}{2r}\ge 1$,
while  $\frac{p}{2}-1 \le 0$.

On the other hand,  for  $r=0$  we have
$-2^{-\frac {p}{s}}+\frac{1}{2^p\cos^p\frac{\pi}{2p}} \ge 0$,
which follows from the Claim 5.1. in the fifth section applied for $p'=\frac{p}{p-1}$ ans $s=p'.$

For  $r=1$  we have  the  inequality
\begin{equation*}
-1+\left(\frac{\cos  \frac{t}{2}}{\cos \frac{\pi}{2p}}\right)^p-\tan\frac{\pi}{2p}\cos\frac{p t }{2} \ge 0,
\end{equation*}
whose proof can be found in \cite{KALAJ.TAMS} or \cite{VERBITSKY.ISSLED}. This is the proof for $s\le  2$.

In the range of $1<p<\frac{5}{4}$ we will se that we can prove the inequality with $s=4$.  This case requires more detailed analysis.
We will proceed in the following way. First, we can repeat a proof line by line until we get that for potential stationary points, we have
\begin{equation*}
\left(r-r^{s+1}\right)\sin\big(t-\frac{tp}{2}\big)+\sin\frac{tp}{2}\left(r^s-r^2\right)\ge 0 .
\end{equation*}

By the Lemma at the end of this section, after dividing both sides of the last inequality by $r-r^{s+1}$, and using the monotonicity in
$r$ we have to prove
\begin{equation*}
\sin\left(t-\frac{tp}{2}\right)-\frac{s-2}{s}\sin\frac{tp}{2}\ge 0 .
\end{equation*}

To prove this inequality, because of $t \in [0,\frac{\pi}{p}],$, we can write $t=\frac{\pi y}{p},$ for $y \in [0,1]$.  So, we will prove
\begin{equation*}
F(y)=\sin\big(\frac{\pi y}{p}-\frac{\pi y}{2}\big)-\frac{s-2}{s}\sin\frac{\pi y}{2}\ge 0 .
\end{equation*}

But, this follows from  $F'(y)=\big(\frac{\pi}{p}-\frac{\pi}{2}\big)\cos\big(\frac{\pi y}{p}-\frac{\pi y}{2}\big)-\frac{\pi(s-2)}{2s}\cos\frac{\pi y}{2} \geq \big(\frac{\pi}{p}-\frac{\pi}{2}\big)\cos\frac{\pi y}{2}-\frac{\pi(s-2)}{2s}\cos\frac{\pi y}{2} =\pi\big(\frac{1}{p}-\frac{s-1}{s}\big)\sin\frac{\pi y}{2}\geq 0,$
for $1<p\leq \frac{5}{4}.$  Hence, $F(y)\geq F(0)=0,$ which was to be shown. In fact, this works for all $p \in (1,2)$ and $s\leq \frac{p}{p-1},$ but our final conclusion depends on the corresponding inequality for $t=0.$

Checking the inequalities for $r=0$ and $r=1$ is similar or the same
as in the first case and we omit its proof. Hence, the only non-trivial part of the proof for the boundary points is for $t=0.$

Dividing both sides of the inequality
\begin{equation*}
-\left(\frac{1+r^4}2\right)^{\frac p4} + \frac{(1+r)^p }{2^p\cos^p\frac{\pi}{2p}} - r^{\frac p2}\tan\frac\pi{2p}\ge 0
\end{equation*}
by  $(1+r)^p$ we get
\begin{equation*}
F(r) = -\left(\frac{1+r^4}{2(1+r)^4}\right)^{\frac  p4} + \frac1{2^p\cos^p\frac{\pi}{2p}} - \left(\frac {\sqrt{r}}{(1+r)}\right)^p
\tan\frac\pi{2p}\ge 0.
\end{equation*}
We find
\begin{equation*}\begin{split}
F'(r)  & = -\frac p4 \left(\frac{1+r^4}{2(1+r)^4}\right)^{\frac p4-1}\frac{4(r^3-1)}{2(1+r)^5}
-p \left(\frac {\sqrt{r}}{1+r}\right)^{p-1}  \frac{\frac 1{2\sqrt{r}}(1+r) -\sqrt{r} }{(1+r)^2} \tan\frac\pi{2p}
\\&= \frac{pr^{\frac p2-1} (1-r)}{2(1+r)^{p+1}} \left( \left(\frac{1+r^4}{2r^2}\right)^{\frac p4-1} (1+r+ r^{-1})
-  \tan\frac\pi{2p}\right).
\end{split}\end{equation*}
Let us prove that for  $1<p\le \frac 54$ the expression inside the brackets is non-positive.
First note that
\begin{equation*}
\left(\frac{r^2+r^{-2} }{2}\right)^{\frac p4-1} (1+r+ r^{-1})\le 3,
\end{equation*}
because for  $r+  r^{-1} = y $ it becomes
\begin{equation*}
\left(\frac{ y^2}2-1\right)^{\frac  p4-1} (1+y) = g(y)\le 3
\end{equation*}
for which we have
\begin{equation*}\begin{split}
g'(y)
= \left( \frac {y^2}2-1\right)^{\frac p4-2}    \left( \left(\frac p4-\frac 12\right)y^2 - \left(\frac p4-1\right)y -1\right)\le0
\end{split}\end{equation*}
since for
\begin{equation*}
h(y) =  \left(\frac p4-\frac 12\right)y^2 - \left(\frac p4-1\right)y -1 = - \left(\frac 12-\frac p4\right) \left(y - \frac{4-p}{4-2p}\right)^2 + \frac{p^2-7p+14}{2-p}
\end{equation*}
we have $h(y)\le h(2) = \frac p2-4<0$, this is because $2\ge \frac{4-p}{4-2p}$ which is satisfied since $p\le \frac 43$.

So, we indeed have
\begin{equation*}
\left(\frac {r^2+r^{-2}}{2}\right)^{\frac p4-1} (1+r+r^{-1})\le 3
\end{equation*}

For $1<p\le \frac 54$ and $\tan \frac\pi {2p}  \ge \tan \frac{2\pi}5 = \sqrt{5+2\sqrt{5}}>3$;
therefore
$F'(r)\le 0$,  i.e. $F$ is monotone decreasing: It follows $F(r) \ge F(1)$.
Since
$2^p F(1) = -1+ \frac 1{\cos^p\frac\pi {2p}}  - \tan \frac \pi {2p}\ge 0$, which follows from
the inequality for $r=1$, we complete the proof for  $1<p\le \frac 54$. (In fact, one see that the conclusion also holds
for $p\leq \frac{\pi}{2\arctan 3}\approx 1,257597$, but we work with $\frac{5}{4},$ as this is just small numerical improvement.)
\end{proof}

Now we can prove our theorem for small $p'$s.

\begin{proof}[Proof of the First Part of Theorem \ref{TH.1}]
If we now  apply the   above lemma    for the      $z=P_+f(\zeta)$ and  $w=P_- f(\zeta)$, we obtain
\begin{equation*}\begin{split}
&-\left\{\frac{|P_+f(\zeta)|^s+|P_-f(\zeta)|^s}{2}\right\}^{\frac ps}+\frac{|P_+f(\zeta)+\overline{P_-f(\zeta)}|^p}
{2^p\cos^p\frac {\pi}{2p}}  -\tan\frac{\pi}{2p} \mathrm{Re} (P_+f(\zeta)P_-f(\zeta))^{ \frac{p}{2}} \ge 0,
\end{split}\end{equation*}
i.e.,
\begin{equation*}\begin{split}
\left(|P_+f(\zeta)|^s+|P_-f(\zeta)|^s\right)^{\frac ps}&\le \frac{2^{\frac {p}{s}}}{2^p\cos^p\frac {\pi}{2p}}|f(\zeta)|^p
-2^{\frac{p}{s}}\tan\frac{\pi}{2p} \mathrm{Re}(P_+f(\zeta)P_-f(\zeta))^{\frac p2}.
\end{split}\end{equation*}
If we integrate the above  inequality over   $\mathbf{T}$,   and if we have in mind  the inequality
\begin{equation*}
\int_{\mathbf{T}}  \Re  (P_+f(\zeta)P_-f(\zeta))^{\frac{p}{2}}|d\zeta| \ge 0
\end{equation*}
(because of subharmonicity) we obtain
\begin{equation*}\begin{split}
\int_{\mathbf{T}} (|P_+ f(\zeta)|^s+|P_-f(\zeta)|^s)^{\frac ps}|d\zeta| \le \frac{ 2^{\frac ps} }{ 2^p\cos^p\frac {\pi}{2p} }
\int_\mathbf{T} |f(\zeta)|^p|d\zeta|,
\end{split}\end{equation*}
which is the statement of the theorem.
\end{proof}

\subsection{The case $p \ge  2$}

Similarly as in the previous case, our proof depends on the certain elementary inequality.
Therefore, we will prove the inequality (with $s=p$):

\begin{lemma}There holds the following estimate:
\begin{equation*} 
\left( \frac {|z|^s + |w|^s}2\right)^{\frac ps}
\le \frac{|z+i\overline {w}|^p }{2^p \sin^p\frac {\pi}{2p}} - (rR)^{\frac p2}
\mathrm {ctg} \frac{\pi}{2p}v_p (u+t),
\end{equation*}
where $u = \arg z$ and $t= \arg w$ and
\begin{equation*} v_p (t)  =
\begin{cases}
-\cos\frac p{2}(\frac \pi2- |t|), & \mbox{if}\  \frac \pi2-\frac {2\pi}p<|t|\le \frac \pi2; \\
\max\{|\cos\frac p{2}(\frac \pi2- t)|, |\cos\frac p{2}(\frac \pi2+t)|\}, & \mbox{if}\ |t|\le \frac \pi2-\frac {2\pi}p.
\end{cases}
\end{equation*}
and  $v_p(t) = v_p (\pi - |t|)$ for $\frac \pi 2\le |t|\le \pi$,
for $p\geq 4$, while for $2\leq p \leq 4,$ we have:
\begin{equation*} 
\left( \frac {|z|^s + |w|^s}2\right)^{\frac ps}
\le \frac{|z+\overline {w}|^p }{2^p \sin^p\frac {\pi}{2p}} - (rR)^{\frac p2}
\mathrm {ctg} \frac{\pi}{2p}v_p (u+t),
\end{equation*}
with $v_p(t)=-\cos\frac{p}{2}(\pi-|t|),$ $|t|\leq \pi.$
\end{lemma}

\begin{proof} \textbf{Case} $p\geq 4$
Because of homogeneity it is enough to prove this inequality for  $z=1$ and
$w= re^{it}$.  Hence, we have to prove
\begin{equation*}
\left(\frac{1 + r^s}2\right)^{p/s}\le \frac{(1+r^2-2r \sin t)^{\frac  p2}}{2^p \sin^p\frac \pi{2p}} - r ^{\frac p2}
\mathrm {ctg} \frac{\pi}{2p}v_p ( t),
\end{equation*}
with $v_p$ as it is defined in the formulation of the lemma.

Note that for $\frac \pi 2 \le|t| \le\pi$ we have
\begin{equation*}\begin{split}
\frac{(1+r^2-2r \sin t)^{\frac p2}}{2^p \sin^p\frac \pi{2p}} - r ^{\frac p2}
\mathrm {ctg} \frac{\pi}{2p}v_p ( t) &\geq \frac{(1+r^2-2r \sin |t|)^{\frac p2}}{2^p \sin^p\frac \pi{2p}} - r^{\frac p2}
\mathrm {ctg} \frac{\pi}{2p} v_p ( |t|)
\\&=   \frac{(1+r^2-2r \sin (\pi-|t|))^{\frac p2}}{2^p \sin^p \frac \pi{2p} } - r ^{\frac p2}
\mathrm {ctg} \frac{\pi}{2p}v_p ( \pi - |t|)
\end{split}\end{equation*}
so it is enough to prove our inequality for $|t|\le \frac \pi 2.$ Also, since $(1+r^2-2r \sin|t|)^{\frac p2}\leq (1+r^2-2r \sin t)^{\frac p2},$ we further reduce to the case $0 \leq t\leq \frac{\pi}{2}.$

Inequality has the form
\begin{equation*}
\left(\frac {1 + r^s}2\right)^{\frac p s}\le \frac{(1+r^2-2r \sin t)^{\frac p2}}{2^p \sin^p\frac\pi{2p}} + r ^{\frac p2}
\mathrm {ctg} \frac{\pi}{2p}  \cos\frac p{2}(\frac \pi2- |t|),
\end{equation*}
for $\frac \pi 2 - \frac {2\pi}p \leq t \leq \frac{\pi}{2},$ i. e. after introducing the change of variables $y=\frac \pi2-t$ we have the inequality
\begin{equation*}
0 \le  - \left(\frac {1 + r^s}2\right)^{\frac ps}+ \frac{(1+r^2-2r \cos t)^{\frac p2}}{2^p \sin^p\frac \pi{2p}} + r ^{\frac p2}
\mathrm {ctg} \frac{\pi}{2p}  \cos\frac {tp }{2}
\end{equation*}
for $0\le t\le \frac{2\pi}p$.

For $t\le \frac \pi 2 - \frac {2\pi}p  $ the  inequality we consider has the form
\begin{equation*}
\left(\frac {1 + r^s}2\right)^{\frac ps}\le
\frac{(1+r^2-2r \sin t)^{\frac p2}}{2^p \sin^p\frac \pi{2p}} -  r ^{\frac p2}
\mathrm {ctg} \frac{\pi}{2p} \max\{|\cos\frac p{2}(\frac \pi2- t)|, |\cos\frac p{2}(\frac \pi2+t)|\}.
\end{equation*}

It is easy to see that for this range of  $t$  the expression
$\frac{(1+r^2-2r \sin t)^{\frac p2}}{2^p \sin^p\frac \pi{2p}}$
has its minimum for  $t=\frac \pi 2-\frac {2\pi}p$, while
\begin{equation*}
-  r ^\frac p2  \mathrm {ctg} \frac{\pi}{2p}
\max\{|\cos \frac p2 (\frac \pi 2- t)|, |\cos\frac p2(\frac \pi 2+t)|\}
\ge  -r ^{\frac p2}  \mathrm {ctg} \frac{\pi}{2p}
\end{equation*}
so our inequality will follow from this one:
\begin{equation*}
\left( \frac {1 + r^s}2\right)^{\frac ps}\le  \frac{(1+r^2-2r \cos \frac{2\pi}{p})^{\frac p2}}{2^p \sin^p\frac \pi{2p}} -  r ^{\frac p2}
\mathrm {ctg} \frac{\pi}{2p},
\end{equation*}
which is exactly the  inequality which we intend to prove in the following lines (for $t=\frac {2\pi}{p}$).

Let us denote
\begin{equation*}
\Phi (r,t)=   - \left( \frac{1 + r^s}2\right)^{\frac ps} +  \frac{(1+r^2-2r \cos t)^{\frac p2}}{2^p \sin^p\frac \pi{2p}} +  r^{\frac p2}
\mathrm {ctg} \frac{\pi}{2p}\cos\frac{tp}{2}.
\end{equation*}

By the above arguments, we have to prove this inequality for $0\le r\le 1$ and $t \in [0,\frac{2\pi}{p}]$.  We have
\begin{equation*}\begin{split}
\frac{\partial\Phi}{\partial r} (r,t) &=  - \frac{p}2 \left(\frac{1 + r^s}2\right)^{\frac ps-1} r^{s-1} +\frac p2
\frac{(1+r^2-2r \cos t)^{\frac p2-1}}{2^p \sin^p \frac \pi{2p}}(2r-2\cos t)
 + \frac p2 r ^{\frac p 2-1} \mathrm {ctg} \frac{\pi}{2p} \cos\frac{ tp}2
\end{split}\end{equation*}
and
\begin{equation*}
\frac{\partial\Phi}{\partial t} (r,t)=  \frac p2 \frac{(1+r^2-2r \cos t)^{\frac p2-1}}{2^p \sin^p\frac \pi{2p}}  2r \sin t- \frac p2
r ^{\frac p2} \mathrm {ctg} \frac{\pi}{2p} \sin \frac{pt}2
\end{equation*}
For potential stationary points, we have $\frac{\partial\Phi}{\partial r}(r,t)= 0 $ and $\frac{\partial\Phi}{\partial t} (r,t)= 0$, from which we get
\begin{equation*}
- \left(\frac {1 + r^s}2\right)^{\frac ps}  =  - \frac {1 + r^s}2 r^{1-s} \left(  r ^{\frac p2-1}
\mathrm {ctg} \frac{\pi}{2p} \cos\frac{ tp}2 \right)
\end{equation*}
Plugging this and the expression for   $\frac{(1+r^2-2r \cos t)^{\frac  p2-1}}{2^p \sin^p\frac \pi{2p}}$
from the equality  $\frac{\partial\Phi}{\partial t}(r,t) =0 $ we have
\begin{equation*}\begin{split}
\Phi (r,t) & =- \frac {1 + r^s}2 r^{1-s}\left(\frac{r^{\frac p2} \mathrm {ctg} \frac \pi{2p}
	\sin \frac {tp}2}{2r \sin t}   (2r-2\cos t )+r^{\frac p2-1} \mathrm {ctg} \frac \pi{2p}
\cos \frac {tp}2\right)
\\&+ (1+r^2-2r\cos t)\frac{r^{\frac p2}\mathrm {ctg} \frac \pi{2p} \sin\frac {pt}2 }{2r \sin t}
+r^{\frac p2}\mathrm {ctg} \frac \pi{2p} \cos \frac {tp} 2\ge 0
\end{split}\end{equation*}
if and only if
\begin{equation*}
-\frac 12 \big (1 + r^{-s}\big) \big(r-\cos t \big)\sin \frac {tp}2 - \frac 12 \big(1+r^{-s}\big) \cos \frac {tp}2
\sin t +   \big(1+r^2-2r \cos t\big)\frac  r2 \sin \frac {tp}2   +\cos\frac { tp}2 \sin t \ge 0.
\end{equation*}
But this can be expressed as
\begin{equation*}
\left(\frac{1+r^2}{2r}\sin \frac {tp}2 -\frac r2 (1+r^{-s})\right) + \frac 1 2 \big(r^{-s}-1\big)
 \big(\cos t \sin \frac {tp}2 - \sin t\cos\frac  {tp}2\big) \ge 0
\end{equation*}
i.e.,
\begin{equation*}
\frac 12 \bigg(\sin \frac {tp} 2 \big(r^{-1}  - r^{1-s}\big) +  \big(r^{-s}-1\big)\sin\big( \big(\frac p 2-1)t \big) \bigg) \ge 0
\end{equation*}

From  $\frac{\partial \Phi}{\partial t} (r,t) = 0$ we see that the condition $\sin \frac {pt}2\geq 0$ is satisfied, therefore
we can rewrite  this as
\begin{equation*}
\sin \big(\frac p2-1\big)t +  \frac {r^s-r^2}{r-r^{s+1}} \sin\frac  {tp}2\ge 0
\end{equation*}
The Lemma at the end of this section gives us $\frac {r^s-r^2}{r-r^{s+1}}\ge - \frac {s-2}{s}$.
Hence it is enough to prove
\begin{equation*}
\sin \big(\frac p2-1\big)t -\frac {s-2}{s} \sin \frac {tp}2\ge 0, \quad  0\le t\le \frac {2\pi}p.
\end{equation*}
We can write $\frac {2\pi y }p$ with $0\le y\le 1$ instead of  $t$, as in the proof in the Lemma 1.
Our estimate now takes the form
\begin{equation*}\begin{split}
\sin (\frac p2-1)\frac {2 \pi y}{p} - \frac {s-2}{s} \sin \pi y  & = \sin (\pi y- \frac 2p\pi y  ) - \frac {s-2}{s}\sin \pi y
\\&\ge (1-\frac 2 p)\sin \pi y -\frac {s-2}{s} \sin \pi y
\\& = (\frac 2s-\frac 2p) \sin \pi y \ge 0,
\end{split}\end{equation*}
where the first inequality follows from the concavity of $\sin-$function on the interval $(0,\pi).$ Hence, our method provide a way to prove the main result even for $s=p$! We will see that our main constraint is the inequality  for $t=0.$

Now, we can treat the boundary points. From the next sequence of the estimates
\begin{equation*}\begin{split}
\frac 2p \frac{\partial \Phi}{\partial t} & = \frac {2r\sin t} {2^ p \sin^p \frac {\pi}{2p} }  (1+r^2-2r\cos t) ^{\frac p2-1}
- r^{\frac p2} \mathrm {ctg} \frac \pi{2p} \sin \frac {tp}2
\\& =r^{\frac p2} \left(\frac {2\sin t} {2^ p \sin^p \frac {\pi}{2p} }  \bigg(\frac{1+r^2}r -2\cos t\bigg)^{\frac p2-1} -  \mathrm {ctg}\frac  \pi{2p} \sin \frac {tp}2  \right)
\\&\ge r^{\frac p2} \left( \frac {2\sin t} {2^ p \sin^p \frac {\pi}{2p} } ( 2- 2\cos t)^{\frac p2-1}-  \mathrm {ctg}\frac  \pi{2p} \sin \frac {tp}2  \right)
\\& \ge r^{\frac p2}\left(\frac { (1-\cos\frac \pi p)^{\frac p2-1}\sin\frac \pi p}{2^{\frac p2}
	\sin^p \frac{\pi}{2p}} - \mathrm {ctg} \frac \pi{2p}   \right)= 0,
\end{split}\end{equation*}
for $t\in [\frac \pi p, \frac {2\pi} p]$, we conclude that   $\frac{\partial \Phi}{\partial t}\ge 0$ for $t\in [\frac \pi p, \frac {2\pi}p]$.
Therefore  $\Phi (r,\frac {2\pi}p)\ge \Phi (r,\frac \pi  p) \ge 0$  (the positivity for $t=\frac{\pi}{p}$ is a consequence from the results of the second section, where  we
established the lower estimates).

For $r=1$ we refer to the  Kalaj's paper (\cite{KALAJ.TAMS}).

For $r= 0 $ we have that the inequality is equivalent with
$ \frac{1}{2}\le \frac 1{(2\sin \frac  \pi {2p})^p},$ which holds by the Claim 5.1 in the fifth section.

All our considerations till this moment holds for $s=p$, so we have just to prove the
inequality for $t=0$ in this case. Its proof is technically involved, so we postpone it for the next section. 

\textbf{Case}  $2\leq p\leq 4$
Again, using homogeneity we reduce the inequality to the case $z=1$ and $w=re^{it}$. Rewriting our inequality in the convenient form (as in the previous case) and making a substitution we arrive at:
\begin{equation*}
\Phi (r,t)=   - \left( \frac{1 + r^s}2\right)^{\frac ps} +  \frac{(1+r^2-2r \cos t)^{\frac p2}}{2^p \sin^p\frac \pi{2p}} +  r^{\frac p2}
\mathrm {ctg} \frac{\pi}{2p}\cos\frac{tp}{2}\geq 0,
\end{equation*}
now with $2\leq p\leq 4$ with $t \in [0,\pi]$ and $r \in [0,1].$ We see that it has the same form as in the previous case. It can be easily reduced to the case $t \in [0,\frac{\pi}{p}], $ by analyzing its partial derivative on $t$ line by the line as for $p\leq 4.$ For the rest of proof we see that all the arguments can be applied again, so it depends on the inequality for $t=0,$ which we give in the next section for $p\geq 2.$
\end{proof}

\begin{proof}[Proof of the Second Part of Theorem \ref{TH.1}]
From the Kalaj's paper(\cite{KALAJ.TAMS})  we know that $\phi_p(z)  = |z|^{\frac p2}v_p(t-\frac \pi2)$ is subharmonic on
$\mathbf{C}$ (see Lemma 4.5 there). Hence, integrating the inequality from the Lemma 4.2 for $z=P_+f(\zeta)$ and $w=P_-f(\zeta),$ we get
\begin{equation*}\begin{split}
\int_{\mathbf{T}}\left(|P_+f(\zeta)|^s+|P_-f(\zeta)|^s\right)^{\frac ps}|d\zeta|&\le \frac{2^{\frac {p}{s}}}{2^p\sin^p\frac {\pi}{2p}}\int_{\mathbf{T}}|f(\zeta)|^p|d\zeta|
\\&-2^{\frac{p}{s}}\mathrm {ctg} \frac{\pi}{2p}\int_{\mathbf{T}}\phi_p (P_+f(\zeta)P_-f(\zeta))|d\zeta|.
\end{split}\end{equation*}
and
$$\int_{\mathbf{T}}\phi_p(P_+f(\zeta)P_-f(\zeta))|d\zeta| \geq 0,$$
thus finally getting:
\begin{equation*}\begin{split}
\int_{\mathbf{T}}\left(|P_+f(\zeta)|^s+|P_-f(\zeta)|^s\right)^{\frac ps}|d\zeta|&\le \frac{2^{\frac {p}{s}}}{2^p\sin^p\frac {\pi}{2p}}\int_{\mathbf{T}}|f(\zeta)|^p|d\zeta|.
\end{split}\end{equation*}
\end{proof}

Here we will formulate and prove the technical lemma which we have used earlier in the proof of the Theorem 1.

\begin{lemma}
	The function $g(r) = \frac {r^s-r^2}{r-r^{s+1}}$ is monotone decreasing in $r$ for $s\ge 2$
\end{lemma}

\begin{proof}
	We find the derivative
	\begin{equation*}\begin{split}
	g'(r)& = \frac { (sr^{s-1} -2r)(r-r^{s+1})- (1- (s+1)r^s) (r^s-r^2) }{ (r-r^{s+1})^2}
	\\&= \frac {(1-s) (r^{2+s} -r^s)+r^{2s}-r^2}{(r-r^{s+1})^2 }
	\end{split}\end{equation*}
	Hence $g'(r)\le  0$ is equivalent with $r^{1-s} - r^{s-1} \ge (s-1)(r^{-1}-r),$  which is equivalent with
	\begin{equation*}
	\frac {r^{1-s} - r^{s-1}}{s-1} - (r^{-1}-r) = \int_r^{r^{-1}} (t^{s-2}-1)dt
	\end{equation*}
	which is $\ge 0$  since
	\begin{equation*}
	\int_r^{r^{-1}} (t^{s-2}-1)dt  = \int_r^{r^{-1}}\frac { (y^{2-s}-1)}{y^2} dy = \int_r^{r^{-1}} (y^{-s}-y^{-2})dy,
	\end{equation*}
	hence
	\begin{equation*}
	\int_r^{r^{-1}} (t^{s-2}-1)dt = \frac 12\int_r^{r^{-1}} (t^{s-2} -1 +t^{-s} - t^{-2})dt  =\frac 12
	\int_r^{r^{-1}}  (t^{-2}  - t^{-s}) (t^s-1)dt \ge 0,
	\end{equation*}
	since $t^{-2}- t^{-s}$ and $t^{s}-1$ are both $\ge 0$ or both  $\le 0$ for $t\ge 1$ and $t\le 1$, respectively.
\end{proof}

\section{Proof of the inequality from Lemma 4.2 for $t=0$}

In this section we will provide a proof of the inequality

\begin{equation}
\label{eq*}
\Phi(r)=-\frac{1+r^p}{2}+\frac{(1-r)^p}{2^p \sin^p\frac{\pi}{2p}}+r^{\frac{p}{2}}\mathrm {ctg}\frac{\pi}{2p}\geq 0.
\end{equation}

For          $r\le 1-2\sin\frac  \pi{2p}$ we have   $(1-r)^p\ge (2\sin\frac \pi {2p} )^p$ and then
$\frac{(1-r) ^p} {2^p\sin^p \pi/2p} \ge 1\ge \frac{1+r^p}{2}$, and our inequality easy follows.
For $ 1-2\sin \frac {\pi}{2p} \le r \le 1$ we will prove that  $r^{\frac p2} \mathrm {ctg}\frac\pi{2p} \ge
\frac {1+r^p}2$ or  $\mathrm {ctg} \frac\pi {2p} \ge\frac {r^{\frac p2} +r^{-\frac p2}}{2}$.   The function
$g(r) = \frac{1}{2} (r^ {\frac p2} + r^{-\frac p2})$  decreases in $r$ and hence its maximum on
$[ 1-2\sin \frac \pi{2p},1]$  may be estimated as
\begin{equation*}
\frac {r^{\frac p 2} + r^{-\frac p2}}{2}
\le  \frac{ (1-2\sin \frac \pi{2p})^{-\frac p2}  +(1-2\sin \frac\pi{2p})^{\frac p2} }{2}
\le \frac{(1- \frac \pi p)^{-\frac p2} +(1-\frac \pi p) ^{\frac p2} }{2},
\end{equation*}
since  $1-2\sin \frac \pi {2p} \ge 1- \frac \pi p $.     So it is enough to verify that
\begin{equation*}
\mathrm {ctg}\frac \pi {2p} \ge\frac{(1-\frac \pi p)^{-\frac p2} +(1-\frac \pi  p) ^{\frac p2} }{2}.
\end{equation*}

Since $\mathrm {ctg} \frac \pi{ 2p}$ is increasing  in $p$ while
$\varphi (p ) = \frac{(1-\frac \pi p)^{-\frac p 2} +(1-\frac \pi p) ^{\frac p2} }{2}$ is
decreasing in $p$ since
$\varphi'(p) = \frac 14 (\log (1-\frac \pi p) + \frac \pi{p-\pi}) )
( e^{\frac p2 \log(1-\frac \pi p)} -  e^{-\frac p2\log(1-\frac \pi p)} )  \le 0$. Here
$\log (1-\frac \pi p) + \frac \pi{p-\pi} $ is equal to  $f(\frac\pi p)$, where
$f(x) = \log (1-x) - \frac x{x-1}$ for which we have  $f'(x) = \frac{x}{(1-x)^2}\ge 0$ and
$f(0)=0$ which gives  $f(x)>0$  fro $x>0$ and consequently $f(\frac \pi p)>0$.  Therefore,
$\mathrm {ctg} \frac\pi { 2p}  -  \frac{(1-\frac \pi p)^{-\frac p 2} +(1-\frac \pi p) ^{\frac p2}}{2}$
is increasing  and its value at $p= 7$ is positive (and is equal $\approx 0,29$), hence we get the result for $ p  \ge 7$.

For $6\leq p\leq 7,$ we use similar argument. In case $r \geq 1-\frac{20}{7p},$ the inequality follows from $\mathrm {ctg}\frac{\pi}{2p}\geq \frac{r^{\frac{p}{2}}+r^{-\frac{p}{2}}}{2}.$Indeed, the right-hand side is decreasing on $r \in [0,1]$ and it is enough to prove
$$\mathrm {ctg}\frac{\pi}{2p}\geq \frac{(1-\frac{20}{7p})^{\frac{p}{2}}+(1-\frac{20}{7p})^{-\frac{p}{2}}}{2}.$$
Imitating the proof for the range $p\geq 7$ we see that the difference between left and right sides is increasing function and its value at 6 is $\mathrm {ctg}\frac{\pi}{12}-\frac{1}{2}\big((\frac{11}{21})^3+(\frac{21}{11})^3\big)\geq 0, $ and our inequality follows. In case $r\leq 1-\frac{20}{7p},$we have:

$$\frac{(1-r)^p}{2^p\sin^p\frac{\pi}{2p}}\geq \bigg(\frac{10}{7p\sin\frac{\pi}{2p}}\bigg)^p\geq \frac{1}{2},$$ which is implied by the following sequence of inequalities:

$$ \bigg(\frac{7p\sin\frac{\pi}{2p}}{10}\bigg)^p \leq \bigg(\frac{7\pi}{20}\bigg)^p\leq \big(\frac{11}{10}\big)^p\leq \big(\frac{11}{10}\big)^7<2.$$

Together with $r^{\frac{p}{2}}\mathrm {ctg}\frac{\pi}{2p}\geq \frac{1}{2}r^p,$ we get $\eqref{eq*}$.

An important step in analysis of the other cases is given in the following claim:

\begin{claim} 
	Function $h(p)=\big(2\sin\frac{\pi}{2p}\big)^{-p}$ is increasing for $p\geq 2.$
\end{claim}

\begin{proof} 
	After the substitution $x=2\sin\frac{\pi}{2p}\leq \sqrt{2},$ we have $p=\frac{\pi}{2\arcsin\frac{x}{2}}$ and the claim reduces to prove that $f(x)=\frac{\pi}{2\arcsin\frac{x}{2}}\log x$ is increasing. Its derivative is $f'(x)=\frac{\pi}{2x\sqrt{1-\frac{x^2}{4}}\arcsin^2{\frac{x}{2}}}\Bigl(\sqrt{1-\frac{x^2}{4}}\arcsin{\frac{x}{2}}-\frac{x}{2}\log x\Bigr),$ so we have to prove the inequality $\sqrt{1-y^2}\arcsin y-y\log(2y)\geq 0, $ for $0\leq y\leq \frac{\sqrt{2}}{2}.$ But, we easily find $g''(y)=-\frac{1}{y(1-y^2)}-\frac{\arcsin y}{(1-y^2)^{\frac{3}{2}}}\leq 0,$ so $g$ is concave and attains it minimum at the end of the interval. From $\lim_{y\rightarrow 0+}g(y)=0$ and $g(\frac{\sqrt{2}}{2})=\frac{1}{4\sqrt{2}}(\pi-\log(4))> 0$ we get the conclusion.
	
	For $5\leq p \leq 6,$ we have:
	
	$$\Phi(r)\geq -\frac{1+r^5}{2}+\frac{(1-r)^6}{(2\sin\frac{\pi}{10})^5}+r^3\mathrm {ctg}\frac{\pi}{10}\geq -\frac{1+r^5}{2}+11(1-r)^6+3r^3=:\frac{1}{2}g_1(r),$$
	where the first inequality follows from the claim and monotonicity of $r^p $ and cotangent function, while the second follows by estimating the values of included functions. We have $g_1'(r)=-5r^4+18r^2-132(1-r)^5$ and $g_1''(r)=4r(9-5r)+660(1-r)^4\geq 0,$ so $g_1$
	is convex. Since $g_1'(0)=-132$ and $g_1'(1)=13,$ it has unique minimum inside the interval $(0,1)$. From $g_1'(\frac{1}{2})>0$ and $g_1'(\frac{49}{100})<0,$ we conclude that the minimum of $g_1$ is in the interval $(\frac{49}{100}, \frac{1}{2})$, therefore we can estimate:
	
	$$2\Phi(r)\geq 6\bigg(\frac{49}{100}\bigg)^3+22\bigg(\frac{1}{2}\bigg)^6-1-\bigg(\frac{1}{2}\bigg)^5=\frac{9197}{500000}>0$$
	
	for $r \in (\frac{49}{100},\frac{1}{2}).$ But this is the interval where $g_1$ has its minimum, so the same inequality holds for $r \in [0,1].$
	
	For $\frac{9}{2}\leq p\leq 5,$ using the same method we get:
	
	$$ \Phi(r)\geq -\frac{1+r^{\frac{9}{2}}}{2}+\frac{(1-r)^5}{(2\sin\frac{\pi}{9})^{\frac{9}{2}}}+r^{\frac{5}{2}}\mathrm {ctg}\frac{\pi}{9}\geq -\frac{1+r^{\frac{9}{2}}}{2}+\frac{11}{2}\big(1-r\big)^5+\frac{5}{2}r^{\frac{5}{2}}=:\frac{1}{2}g_2(r). $$
	
	Here we find $g_2'(r)=-\frac{9}{2}r^{\frac{7}{2}}-55(1-r)^4+\frac{25}{2}r^{\frac{3}{2}}$ and $g_2''(r)=\frac{1}{4}r^{\frac{1}{2}}(75-63r^2)+220(1-r)^3 \geq 0,$$ g_2$ is convex and from $g_2'(\frac{12}{25})<0$ and $g_2'(\frac{49}{100})>0;$ this implies that the minimum is in $(\frac{12}{25},\frac{49}{100})$ and we have:
	
	$$g_2(r)\geq -1-\bigg(\frac{49}{100}\bigg)^{\frac{9}{2}}+11\bigg(\frac{51}{100}\bigg)^5+5\bigg(\frac{12}{25}\bigg)^{\frac{5}{2}} >0$$
	and consequently $\Phi(r)>0.$
	
	For $4\leq p\leq \frac{9}{2} ,$ we get:
	
	$$ \Phi(r)\geq -\frac{1+r^4}{2}+\frac{(1-r)^{\frac{9}{2}}}{(2\sin\frac{\pi}{8})^4}+r^{\frac{9}{4}}\mathrm {ctg}\frac{\pi}{8}\geq -\frac{1+r^4}{2}+\frac{29}{10}\big(1-r\big)^5+\frac{12}{5}r^{\frac{9}{4}}=:\frac{1}{10}g_3(r). $$
	
	Now, $g_3'(r)=-20r^3-\frac{261}{2}(1-r)^{\frac{7}{2}}+54r^{\frac{5}{4}}$ and $g_3''(r)=\frac{1}{2}r^{\frac{1}{4}}(135-120r^{\frac{3}{4}})+\frac{1827}{4}(1-r)^{\frac{5}{2}} > 0,$$ g_3$ is convex and from $g_2'(\frac{43}{100})<0$ and $g_2'(\frac{11}{25})>0;$ which gives us that the minimum is in $(\frac{43}{100},\frac{1}{25})$ and:
	
	$$g_3(r)\geq -5-5\bigg(\frac{11}{25}\bigg)^4+29\bigg(\frac{14}{25}\bigg)^{\frac{9}{2}}+24\bigg(\frac{43}{100}\bigg)^{\frac{9}{4}} >0$$
	and  $\Phi(r)>0.$
	
	For $\frac{7}{2}\leq p\leq 4,$ we get:
	
	$$ \Phi(r)\geq -\frac{1+r^{\frac{7}{2}}}{2}+\frac{(1-r)^4}{(2\sin\frac{\pi}{7})^{\frac{7}{2}}}+r^2\mathrm {ctg}\frac{\pi}{7}\geq -\frac{1+r^{\frac{7}{2}}}{2}+\frac{8}{5}\big(1-r\big)^4+2r^2=:\frac{1}{10}g_4(r). $$
	
	Here, we have $g_4'(r)=-\frac{35}{2}r^{\frac{7}{2}}+40r-64(1-r)^3,$ $g_4''(r)=-\frac{175}{4}r^{\frac{3}{2}}+40+192(1-r)^2,$ and $g_4^{(3)}(r)=384r-\frac{525}{8}\sqrt{r}-384\leq 0$ on $(0,1)$, thus $ g_4'$ is concave and, since $g_4'(0)<0, g_4'(1)>0, g_4''(0)>0,g_4''(1)<0$ we conclude that $g_4'$
	increases till some $r_0$ and then decreases, having exactly one zero in $(0,1).$ Also, $g_4'(\frac{79}{200})<0, g_4'(\frac{99}{250})>0,$ therefore:
	$$g_4(r)\geq -5-5\bigg(\frac{99}{250}\bigg)^{\frac{7}{2}}+20\bigg(\frac{79}{200}\bigg)^2+16\bigg(\frac{151}{250}\bigg)^4 >0,$$
	i.e.  $\Phi(r)>0.$
	
	For $\frac{13}{4}\leq p\leq \frac{7}{2},$ we get:
	
	$$ \Phi(r)\geq -\frac{1+r^{\frac{13}{4}}}{2}+\frac{(1-r)^{\frac{7}{2}}}{(2\sin\frac{\pi}{6})^{3}}+r^{\frac{7}{4}}\mathrm {ctg}\frac{2\pi}{13}\geq -\frac{1+r^{\frac{13}{4}}}{2}+\big(1-r\big)^{\frac{7}{2}}+\frac{19}{10}r^{\frac{7}{4}}=:\frac{1}{10}g_5(r). $$
	
	We easily find  $g_5'(r)=-\frac{65}{4}r^{\frac{9}{4}}+\frac{133}{4}r^{\frac{3}{4}}-35(1-r)^{\frac{5}{2}},$ $g_5''(r)=\frac{1}{16r^{\frac{1}{4}}}(399-585r^{\frac{3}{2}})+\frac{175}{2}(1-r)^{\frac{3}{2}},$ thus for $r\leq r_0=(\frac{133}{195})^{\frac{2}{3}}$ the function $g_5$ is convex. For $r \geq r_0,$ we have $g_5'(r)=\frac{r^{\frac{3}{4}}}{4}(133-65r^{\frac{3}{2}})-35(1-r)^{\frac{5}{2}}\geq \frac{68r_0^{\frac{3}{4}}}{4}-35(\frac{23}{100})^{\frac{5}{2}}=17\sqrt{\frac{133}{195}}-35(\frac{23}{100})^{\frac{5}{2}}>0,$ hence, in this case, $g_5$ is increasing (Here we use that $r_0>\frac{77}{100}$). We can now conclude that the minimum of $g_5$ on $(0,1)$ is achieved in the interval $(0, r_0),$ so, from $g_5'(\frac{163}{500})g_5'(\frac{327}{1000})<0,$ we can estimate $g_5$ as:
	
	$$g_5(r)\geq -5-5\bigg(\frac{327}{1000}\bigg)^{\frac{13}{4}}+10\bigg(\frac{673}{1000}\bigg)^{\frac{7}{2}}+19\bigg(\frac{163}{500}\bigg)^{\frac{7}{4}} >0,$$
	i.e.  $\Phi(r)>0.$
	
	For $3 \leq p\leq \frac{13}{4},$ we get:
	
	$$ \Phi(r)\geq -\frac{1+r^3}{2}+\frac{(1-r)^{\frac{13}{4}}}+r^{\frac{7}{4}}\mathrm {ctg}\frac{\pi}{3}\geq -\frac{1+r^3}{2}+\big(1-r\big)^{\frac{13}{4}}+\frac{17}{10}r^{\frac{13}{8}}=:\frac{1}{10}g_6(r). $$
	
	We easily find  $g_6'(r)=-15r^2-\frac{65}{2}(1-r)^{\frac{9}{4}}+\frac{221}{8}r^{\frac{5}{8}},$ $g_6''(r)=r^{\frac{3}{8}}(\frac{1105}{64}-30r^{\frac{5}{8}})+\frac{585}{8}(1-r)^{\frac{5}{4}},$  hence, for $r^{\frac{5}{8}}<\frac{221}{384},$ which holds for $r<\frac{2}{5}$ this function is convex. For $r\geq \frac{2}{5},$ we have $g_6'(r)=r^{\frac{5}{8}}(\frac{221}{8}-15r^{\frac{3}{8}})-\frac{65}{2}(1-r)^{\frac{13}{4}}\geq (\frac{2}{5})^{\frac{5}{8}}\cdot\frac{101}{8}-\frac{65}{2}(\frac{3}{5})^{\frac{13}{4}}>0,$ and $g_6$ is increasing on $[\frac{2}{5},1).$ Again, we conclude that it has unique point of minimum in $(0,1)$ and, from $g_6'(\frac{347}{1000})g_6'(\frac{87}{250})<0,$ therefore:
	
	$$g_6(r)\geq -5-5\bigg(\frac{87}{250}\bigg)^3+10\bigg(\frac{163}{250}\bigg)^{\frac{13}{4}}+17\bigg(\frac{347}{1000}\bigg)^{\frac{13}{8}} >0.$$
	
	For $ \frac{29}{10} \leq p\leq 3,$ we get:
	
	$$ \Phi(r)\geq -\frac{1+r^{\frac{29}{10}}}{2}+\frac{(1-r)^3}{(2\sin{\frac{5\pi}{29}})^{\frac{29}{10}}}+ r^{\frac{3}{2}}\mathrm {ctg}{\frac{5\pi}{29}}\geq -\frac{1+r^{\frac{29}{10}}}{2}+\frac{9}{10}\big(1-r\big)^3+\frac{8}{5}r^{\frac{3}{2}}=:\frac{1}{10}g_7(r). $$
	
	Since $g_7^{(3)}(r)=-54-6r^{-\frac{3}{2}}-\frac{2259}{200}r^{-\frac{1}{10}}<0,$  $g_7'(r)=-\frac{29}{2}r^{\frac{19}{10}}-27(1-r)^2+24r^{\frac{1}{2}}$ is concave and $g_7'(0)<0, g_7'(1)>0$ $g_9$ has exactly one zero. Again, the estimate for $g_7$ comes from $g_7'(\frac{33}{100})g_7'(\frac{1}{3})<0$:
	
	$$g_7(r)\geq -5-5\bigg(\frac{1}{3}\bigg)^{\frac{29}{10}} +9\bigg(\frac{2}{3}\bigg)^3+16\bigg(\frac{33}{100}\bigg)^{\frac{3}{2}}>0.                     $$
	
	For $\frac{27}{10}\leq p \leq \frac{29}{10},$
	
	$$ \Phi(r)\geq -\frac{1+r^{\frac{27}{10}}}{2}+\frac{(1-r)^{\frac{29}{10}}}{(2\sin{\frac{5\pi}{27}})^{\frac{27}{10}}}+ r^{\frac{29}{20}}\mathrm {ctg}{\frac{5\pi}{27}}\geq -\frac{1+r^{\frac{27}{10}}}{2}+\frac{77}{100}\big(1-r\big)^{\frac{29}{10}}+\frac{3}{2}r^{\frac{29}{20}}=:\frac{g_8(r)}{100}. $$
	
	Considering the function $h(r)=3r^{\frac{29}{20}}-1-r^{\frac{27}{10}}$ we see that it is increasing, since $h'(r)=\frac{r^{\frac{9}{20}}}{20}(87-54r^{\frac{5}{4}})>0,$ so we can see that for $r\geq \frac{17}{20}$ we have $h(r)\geq h(\frac{17}{20})>0$ and consequently $g_8>0.$ On the interval $[0,\frac{17}{20}]$ $g_8$ is convex, which follows from $g_8''(r)=\frac{r^{\frac{1}{20}}}{8}(1635-1836r^{\frac{3}{4}})+\frac{42427}{10}(1-r)^{\frac{9}{10}}$ and positivity of the first summand on this interval. Thus, $g_8$ has unique point of minimum, which is estimated from $g_8'(\frac{31}{100})g_8'(\frac{311}{1000})<0$ as:
	
	$$g_8(r)\geq -50-50\bigg(\frac{311}{1000}\bigg)^{\frac{27}{10}} +77\bigg(\frac{689}{1000}\bigg)^{\frac{29}{10}}+150\bigg(\frac{31}{100}\bigg)^{\frac{29}{20}}>0. $$
	
	For $\frac{5}{2}\leq p \leq \frac{27}{10},$
	
	$$ \Phi(r)\geq -\frac{1+r^{\frac{5}{2}}}{2}+\frac{(1-r)^{\frac{27}{10}}}{(2\sin{\frac{\pi}{5}})^{\frac{5}{2}}}+ r^{\frac{27}{20}}\mathrm {ctg}{\frac{\pi}{5}}\geq -\frac{1+r^{\frac{5}{2}}}{2}+\frac{2}{3}\big(1-r\big)^{\frac{27}{10}}+\frac{172}{125}r^{\frac{27}{20}}=:\frac{g_9(r)}{750}. $$
	
	The function $g_9$ is convex on $[0,\frac{19}{25}],$ since $g_9''(r)=r^{-\frac{3}{20}}(\frac{59211}{50}-\frac{5625}{4}r^{\frac{13}{20}})+2295(1-r)^{\frac{7}{10}}$ and the expression inside the first brackets is positive on this interval. Thus, on this part of the interval $[0,1]$ we can use arguments as earlier. For other values of we will see that $g_9$ is positive, using the function $h(r)=\frac{344}{125}r^{\frac{27}{20}}-1-r^{\frac{5}{2}}$. It is increasing $(h'(r)=r^{\frac{7}{20}}(\frac{2322}{625}-\frac{5}{2}r^{\frac{23}{20}})),$hence $h(r)\geq h(\frac{19}{25})>0.$ Now, from $g_9'(\frac{347}{1000})g_9'(\frac{87}{250})<0, $ we see that we can estimate as:
	
	$$g_9(r)\geq -375-375\bigg(\frac{87}{250}\bigg)^{\frac{5}{2}} +500\bigg(\frac{163}{250}\bigg)^{\frac{27}{10}}+1032\bigg(\frac{347}{1000}\bigg)^{\frac{27}{20}}>0. $$

	Finally, in the case $2\leq p\leq \frac{5}{2},$ we use Jensen's inequality
	\begin{align*}
	\frac{(1-r)^p}{2^p \sin^p\frac{\pi}{2p}}+r^{\frac{p}{2}}\mathrm {ctg}\frac{\pi}{2p}&\geq \bigg(1+\mathrm {ctg}\frac{\pi}{2p}\bigg)^{1-\frac{p}{2}}\bigg(\frac{(1-r)^2}{4\sin^2\frac{\pi}{2p}}+r\mathrm {ctg}\frac{\pi}{2p}\bigg)^{\frac{p}{2}}\\
	&= \bigg(1+\mathrm {ctg}\frac{\pi}{2p}\bigg)^{1-\frac{p}{2}}\bigg(\mathrm {ctg}\frac{\pi}{2p}\bigg)^{\frac{p}{2}}\bigg(\frac{(1-r)^2}{2\sin\frac{\pi}{p}}+r\bigg)^{\frac{p}{2}}\\
	&\geq \bigg(1+\mathrm {ctg}\frac{\pi}{2p}\bigg)^{1-\frac{p}{2}}\bigg(\mathrm {ctg}\frac{\pi}{2p}\bigg)^{\frac{p}{2}}\bigg(\frac{1+r^2}{2}\bigg)^{\frac{p}{2}},
	\end{align*}
	thus, it is enough to prove that the last expression is $\geq \frac{1+r^p}{2},$ or equivalently:
	$$ 2^{1-\frac{p}{2}}\bigg(1+\mathrm {ctg}\frac{\pi}{2p}\bigg)^{1-\frac{p}{2}}\bigg(\mathrm {ctg}\frac{\pi}{2p}\bigg)^{\frac{p}{2}}\geq 1,$$
	having in mind that $(1+r^p)^{\frac{1}{p}}\leq (1+r^2)^{\frac{1}{2}}$ holds for $p\geq2.$
	Taking the logarithms in the last inequality and changing the variable as $\mathrm {ctg}\frac{\pi}{2p}=x \in [1,\mathrm {ctg}\frac{\pi}{5}],$ we have to prove:
	
	$$f(x):=4\log(2+2x)\arctan x+\pi\log(\frac{x}{2+2x}) \geq 0.$$
	
	Since $f(1)=0$ and $f(\mathrm {ctg}\frac{\pi}{5})>0,$ we get the conclusion after providing a proof that $f$ is concave. But $f''(x)=-\frac{4\arctan x}{(x+1)^2}+\frac{8x\log(2+2x)}{(x^2+1)^2}-\frac{8}{(x^2+1)(x+1)}-\frac{\pi(2x+1)}{x^2(x+1)^2}.$ We will prove  $\frac{8x\log(2+2x)}{(x^2+1)^2}-\frac{8}{(x^2+1)(x+1)}-\frac{\pi(2x+1)}{x^2(x+1)^2}\leq 0,$ which gives us the desired implication.
	From $\log(2+2x)< \log 5<\frac{5}{3},$ we see that $\frac{5}{3}\leq \frac{8x^2(x+1)(x^2+1)+3(2x+1)(x^2+1)^2}{8x^3(x+1)^2} \leq \frac{8x^2(x+1)(x^2+1)+\pi(2x+1)(x^2+1)^2}{8x^3(x+1)^2}$ will be enough. While the second inequality is trivial, the first one is equivalent with $P(x)=2x^5-47x^4+20x^3+42x^2+18x+9\geq 0.$ Since $40\big(\frac{7}{5}\big)^3+120\cdot\frac{7}{5}+84<564,$ we have that $P''(x)=40x^3-564x^2+120x+84<0,$ and hence, $P'(x)=10x^4-188x^3+60x^2+84x+18$ is decreasing on $(1,\frac{7}{5})$, thus giving $P'(x)\leq P'(1)<0,$ i. e. $P(x)$ is decreasing on the same interval and consequently $P(x)\geq P(\frac{7}{5})>0.$ (Here we use that $\mathrm {ctg}\frac{\pi}{5}<\frac{7}{5}.$)
\end{proof}

\section{Riesz's theorem on conjugate harmonic  functions for various function spaces}

The classical  Riesz's theorem on  conjugate harmonic functions  says that for every  $p\in (1,\infty)$ there exists  a constant  $c_p$ such that
\begin{equation*}
\|U+  i V\|_p\le c_p \|U\|_p,
\end{equation*}
where $U$ is a real-valued   function in $h^p$, $V$ is a harmonic  conjugate  to $U$ on $\mathbf{U}$,  normalized such that $V(0)=0$.    See,  for
instance,   \cite[Theorem 17.26]{RUDIN.BOOK}. Verbitsky proved \cite{VERBITSKY.ISSLED} that the best possible  constant in the Riesz inequality is
\begin{equation}\label{EQ.VC}
c_p =
\begin{cases}
\mathrm {sec} \frac \pi{2p}, & \mbox{if}\  1<p\le 2; \\
\mathrm {csc} \frac \pi{2p}, & \mbox{if}\  2\le p<\infty.
\end{cases}
\end{equation}

\subsection{}
Here we prove the following generalisation of the M. Riesz theorem on harmonic  conjugate functions. For example, the next  theorem incudes analytic
functions that belong to well known Fock spaces. Fock spaces are obtained  if  we take     the measure $\omega (z)= e^{ -\alpha |z|^2 }dA(z)$, where
$\alpha>0$ is a constant. We deduce  also  the M. Riesz theorem for Besov type spaces.

\begin{theorem}
Let $X$ be a set (for example, it may be a  space with a norm). Assume that $A: X\rightarrow A_p(\omega)$ is a mapping (an operator, not necessarily linear),  where $A_{p,\omega}$ is the space of all analytic functions in the domain $D$ which is the unit disc $\mathbf{U}$ or the complex plane
$\mathbf{C}$.   Let
\begin{equation*}
A_p(\omega) = \{ f : \| f \|_{p,\omega}\ \text{is finite} \},
\end{equation*}
where
\begin{equation*}
\|f\|_{p,\omega} = \left\{\int_D  |f(\zeta )|^p\omega(\zeta)d\zeta\right\}^{\frac 1p};
\end{equation*}
here  $\omega (\zeta)$ is radially symmetric weight  function, i.e., we have $\omega (\zeta)  =  \omega (|\zeta|)$.

Assume that the operator $\mathrm {Re} A$ is bounded; it  maps $X$ into $a_p(\omega)$ the space of real parts of functions in $A_p (\omega)$. Then
the operator $A$ is also  bounded. Moreover, there exists  a  universal  constant estimate.
\end{theorem}

\begin{proof}
We have     the following inequality
\begin{equation*}
|z|^p\le c_p^p  |\mathrm {Re}  (z) |^p -  b_p s(z),
\end{equation*}
where $b_p$  is a positive constant   and  $s(z)$  is a subharmonic function,  and
\begin{equation}\label{EQ.VC}
c_p =
\begin{cases}
\mathrm {sec} \frac \pi{2p}, & \mbox{if}\  1<p\le 2; \\
\mathrm {csc} \frac \pi{2p}, & \mbox{if}\  2\le p<\infty.
\end{cases}
\end{equation}

For $x\in X$, let  $ Ax\in A_p(\omega)$ be a function denoted by $f (\zeta)$,   $\zeta\in D$. Applying the  inequality  above we obtain
\begin{equation*}
|f(\zeta)|^p\le c_p^p  |\mathrm {Re} f(\zeta)|^p - b_ps( f(\zeta))
\end{equation*}
Assume  that $\Re f(0)\ge0$ (otherwise we may  consider $-f(\zeta)$). If we multiply by $\omega(\zeta)$  and then integrate the above
inequality in $D$ we obtain
\begin{equation*}
\int _D |f(\zeta)|^p  \omega(\zeta )d\zeta \le c_p^p  \int_D |\mathrm {Re} f(\zeta)|^p \omega(\zeta )d\zeta
- b_p\int_D s( f(\zeta)) \omega(\zeta )d\zeta
\end{equation*}
Since $s(f(\zeta))$   is a subharmonic function on $D$, by  using polar coordinates and  the mean value theorem,  we obtain  (the symbol
$a$ stand   for $1$ or $\infty$)
\begin{equation*}\begin{split}
\int_D s (f(\zeta)) \omega(\zeta )d\zeta  & = \int_0^{a} r dr \int_{r \mathbf{T}} s( f(\zeta))
 \omega(|\zeta| )|d\zeta|  \ge s (f(0))  \cdot \int_0^a \omega(r) rdr\ge 0.
\end{split}\end{equation*}
It follows
\begin{equation*}
\int _D |f(\zeta)|^p  \omega(\zeta )d\zeta \le c_p^p  \int_D |\mathrm {Re}  f(\zeta)|^p \omega(\zeta )d\zeta.
\end{equation*}
Now, we  have   the inequality we need
\begin{equation*}
 \|f  \|_{p,  \omega} \le c_p  \|\mathrm {Re} f\|  _{p,  \omega},\quad \text{i.e.},\quad   \|Ax\|_{p,  \omega} \le c_p  \|\mathrm {Re} Ax\|
   _{p,  \omega}.
\end{equation*}

If   $X$ is a space with norm,  then according to the obtained inequality  we have
\begin{equation*}
\frac{ \|Ax\|_{p,  \omega}}{\|x\|_X} \le c_p  \frac{\|\mathrm {Re}  Ax\|  _{p,  \omega}}{\|x\|_X},\quad x\ne 0.
\end{equation*}
Therefore, if we take  $\sup_{x\ne 0}$ we obtain
$\|A:X\rightarrow   A_p (\omega)\|\le c_p   \|\mathrm {Re}  A :X\rightarrow   a_p (\omega) \|$, which we aimed to prove.
\end{proof}

\begin{remark}
Of course, one can recover the classical M. Riesz theorem for classical   Hp-spaces  from the above   theorem. If we can consider the Bergman spaces
$A_{p,\alpha}$,  $\alpha>-1$,  and   identity operator,  and if we   let $\alpha \rightarrow -1$,   we obtain M. Riesz theorem for  the Hardy space
$H^p$. Indeed, we have $\|\mathrm {Id} (f) \|_{A_{p,\alpha} } \le c_p  \|\mathrm {Re} (\mathrm {Id}) (f)\|_{h_{p,\alpha}}$. In other  words we have
$  \|  f  \|_{A_{p,\alpha} } \le c_p \|\mathrm {Re}   (f)\|_{a_{p,\alpha}} $ for every $f\in A_{p,\alpha}$.   If we let  $\alpha\rightarrow -1$, we
obtain the M. Riesz theorem  on conjugate harmonic  functions in  Hardy space  $H^p$.  Note that this is possible  because of the universality  of
the constant $c_p$, so  we  can take $\alpha\rightarrow -1$ and this constant  remains in the inequality.
\end{remark}

\subsection{}
Let
\begin{equation*}
B_{p,\omega} = \{f: \int_{D} |f'(\zeta)|^p\omega (\zeta) d\zeta<\infty\}
\end{equation*}
be the weighted  Besov type  space,  and let $b_{p,\omega}$  be the  harmonic   Besov space.

\begin{theorem}[M. Riesz theorem for weighted   Besov spaces]
For  $f\in B_{p,\omega}$  we have
\begin{equation*}
\|f\|_{B_{p,\omega}}\le c_p  \|\mathrm {Re} f \|_{b_{p,\omega}}.
\end{equation*}
\end{theorem}

\begin{proof}
The  proof  goes in the  similar way  as the  proof of the  above   theorem if we have in mind  the relation  $\mathrm {Re} f'(z) = \nabla
\mathrm {Re} f(z)$.  Indeed, if in the inequality $|z|^p \le c_p^p  |\mathrm {Re} z|^p -  b_p s(z)$, set $f'(\zeta)$ we obatin
\begin{equation*}
|f'(\zeta)|^p\le c_p^p  |\mathrm {Re}  f'(\zeta)|^p - b_p s ( f'(\zeta) )
\end{equation*}
i.e. if we use the above relation we obtain
\begin{equation*}
|f'(\zeta)|^p\le c_p^p   |\nabla \mathrm {Re} f(\zeta)|^p - b_p s ( f'(\zeta) )
\end{equation*}
Now,  if we integrate with respect  to the $D$ as in the  preceding theorem  we obtain
\begin{equation*}
\|f\|_{B_{p,\omega}}\le c_p  \|\mathrm {Re}  f \|_{b_{p,\omega}},
\end{equation*}
which is the   statement of our theorem.
\end{proof}

Riesz's theorem for the   Bergman spaces was proves by Forelli and Rudin in 1974 and for the  Bloch space recently   by Kalaj and Markovi\'{c}
\cite{KALAJ.MARKOVIC.MS}.

\subsection{}
Here  prove M. Riesz theorem on conjugate real-valued harmonic function in a domain  of $\mathbf{C}$   that belong to so called Lumer's Hardy spaces.
We will  say something  on that  spaces firstly.

There  are  generalizations of Hardy spaces for  arbitrary domains in $\mathbf{C}$.  The generalizations  we consider in this  subsection are known
as  Lumer's Hardy spaces  \cite{DUREN.BOOK, FISHER.BOOK, LUMER.PARIS, LUMER-NAIM, RUDIN.TAMS}. The harmonic Lumer's Hardy  space $(Lh)^p(\Omega)$
contains all harmonic  functions  $U$  on a domain $\Omega\subseteq \mathbf{C}$ such that  the subharmonic function $|U|^p$ has a harmonic majorant
on $\Omega$.  In that case, the  function  $|U|^p$ has the  least harmonic majorant on $\Omega$. Let it be denoted by  $H_U$. For $\zeta_0\in\Omega$
one introduces a  norm  on  $(Lh)^p (\Omega)$  in the  following way
\begin{equation}\label{EQ.LN}
\|U\|_ {p,\zeta_0} = H_U ^{1/p} (\zeta_0).
\end{equation}
It is known that the   different  norms on  $(Lh)^p(\Omega)$  that arise by  selecting different elements of the  domain $\Omega$ are mutually equivalent.
The analytic Lumer's Hardy  space   $(LH)^p(\Omega)$ is the subspace  of  $(Lh)^p(\Omega)$   that  consists of all analytic functions. The two spaces
$(Lh)^p(\mathbf{U})$ and $h^p$  coincide (as do $(LH)^p(\mathbf{U})$ and $H^p$). The norms on these spaces are the same, if we select $\zeta_0 =0$ for the
Lumer  case. Moreover,  if $\Omega$ is a simply connected domain such that  $\partial\Omega$ is  a Jordan curve which is sufficiently smooth, the Lumer's
Hardy space   $(LH)^p (\Omega)$ is the same as the  Smirnov's Hardy space $E^p(\Omega)$ which  is defined by requirement that the integral means of an
analytic  function over   certain  family  of curves    in the domain $\Omega$   remains  bounded -- we  refer to the  tenth  chapter in \cite{DUREN.BOOK}.
Therefore, we have $(LH)^p(\mathbf{U})=E^p(\mathbf{U})= H^p$. This follows  from  Theorem 10.2 in  \cite{DUREN.BOOK}  which  gives a sufficient criterion
for coincidence of analytic Lumers's and   Smirnov's Hardy  spaces. This criterion  may  be  easily adapted in order  to conclude that the two types of
harmonic  spaces in a sufficiently smooth  domain  coincides.

Note that  Lumer's Hardy  spaces are conformally invariant. In other words, if   $\Phi$ is a conformal mapping of a domain $\tilde{\Omega}$ onto $\Omega$,
then   $U\in   (Lh)^p(\Omega)$ if and only if  $\tilde{U}=U\circ\Phi \in Lh)^p (\tilde{\Omega})$. The mapping $\Phi$ induces  an  isometric  isomorphism
$U\rightarrow\tilde{U}$ of the space $(Lh)^p(\Omega)$ onto $(Lh)^p (\tilde{\Omega})$, since the equality for the least harmonic majorants $H_U \circ\Phi
= H_{\tilde{U}}$ implies that $\|U\|_{p,\zeta_0}=\|\tilde{U}\|_{p,\tilde{\zeta_0}}$, where $\tilde{\zeta_0} \in \tilde{\Omega}$ satisfies  $\zeta_0=
\Phi (\tilde{\zeta_0})$.

We have   conjectured in \cite{MARKOVIC.AMM} that the M. Riesz  theorem with  the   Verbitsky   constant  $c_p$
is valid  in the case of Lumer's Hardy spaces  $(LH)^p(\Omega)$  for every   $p\in (1,\infty )$  and    $\Omega
\subseteq\mathbf{C}$. However, we prove this   conjecture  for analytic functions with  the positive real  part.
Our  aim in this section  is to prove the Riesz       theorem for real-valued            harmonic functions  in
the Lumer's Hardy space $(Lh)^p(\Omega)$ for  which  there  exists a conjugate  with the best possible constant
it  the case of positive harmonic functions on $\Omega$.
Because  of duality we consider the case $1<p<2$.  This is the content  of the  following theorem.

\begin{theorem}\label{TH.RIESZ.LUMER}
Let  $\Omega\subseteq  \mathbf{C}$ be a domain and $\zeta_0\in \Omega$. Assume  that for a positive  $U \in
(Lh)^p(\Omega)$  there exists a harmonic conjugate of $U$  on the  domain $\Omega$, denoted by $V$, and let
it be normalized such that $V(\zeta_0)=0$. Then  we   have the  Riesz  inequality
\begin{equation}\label{EQ.R}
\|U+iV\|_{p,\zeta_0}\le   c_ p\|U\|_{p,\zeta_0},
\end{equation}
where  the best  possible  constant is the Verbitsky constant if $U$ is positive.
\end{theorem}

\begin{proof}
In \cite{MARKOVIC.AMM} the case $p =2$ is considered. Now, we can adapt the   proof for $p=2$ given there
\cite{MARKOVIC.AMM}.  Recall that  in  \cite{MARKOVIC.AMM}   we have used  the following elementary equality,
which is  easy to check:
\begin{equation}\label{EQ.Z}
|z|^2  = 2 (\Re z)^2 - \Re z^2,\quad z\in\mathbf{C}.
\end{equation}
For     $p\ne 2$ we should use an inequality in order     to replace  the  equality  given above.

If  $1<p<2$  for     $c_p^p  =  ( {\cos \frac \pi{2p}} )^{-p} $ we have the  inequality
\begin{equation}\label{EQ.EQP}
|z|^p\le c_p^p  (\Re z) ^p - h(z),
\end{equation}
where    $h(z)$ is a   harmonic  function in the domain  $\{z:\Re z >0 \}$, which satisfies $h(x) >0 $ for  every
$x>  0$; actually, the function  $h $  is given by $h(z) = \Re  (z^p)$.
For the proof see the Kalaj's paper \cite{KALAJ.TAMS}, or the  Verbitsky work.

Assume first that $U>0$.
Let   the analytic function    $U + iV$ be denoted  by  $F$, and let $H_U$ be  the  least harmonic majorant
of  the subharmonic function  $|U |^p$ on  $\Omega$. By applying inequality  \eqref{EQ.EQP} for $z  = F(\zeta)$,
$\zeta\in \Omega$,                     we obtain
\begin{equation*}\begin{split}
|F(\zeta)|^p&\le c_p^p  (\Re F(\zeta))^2 - h (F(\zeta))=  c_p^p  |U(\zeta) | ^2 -  h (F(\zeta))
 \\& \le c_p^p H_U (\zeta) -    h (F(\zeta)),
\end{split}\end{equation*}
which proves that $c_p^p H_U (\zeta) -    h (F(\zeta))$   is a harmonic   majorant  of  $| F |^p$ on $\Omega$.
It follows that $F\in (LH)^p(\Omega)$. Moreover, if $H_F$ is the least harmonic majorant     of  $|F|^p$ on
$\Omega$, we have
\begin{equation*}
H_F (\zeta)   \le c_p^p H_U (\zeta) -    h (F(\zeta)).
\end{equation*}
Since $F(\zeta_0) = U(\zeta_0)$ is a positive   real  number,  we  obtain
\begin{equation*}\begin{split}
\|F\|^p_{p,\zeta_0} & = H_F(\zeta_0)  \le  c_p^p H_U (\zeta_0 ) -    h (F(\zeta_0))
 \le c_p^p H_U(\zeta_0)   =  c_p^p   \|U\|^2_{p,\zeta_0}.
\end{split}\end{equation*}
Finally, we conclude that
\begin{equation*}
\|F\|_{p,\zeta_0} \le c_p \|U\|_{p,\zeta_0},
\end{equation*}
which is what we wanted to prove.

The sharpness of $c_p$ in this case   follows from the Verbitsky  result.
\end{proof}

We give now here a proof of the Riesz inequality  but  with a constant   different   from   the conjectured one.
We use the similar   approach as in the  Pavlovi\'{c} paper \cite{PAVLOVIC.PAMS}. Assume that   the boundary of
the domain $\partial \Omega$ is   smooth. Then  $U$ at the boundary may be  decomposed as $U ^*= U_1 ^*- U_2^*$,
where $U_1^*$ and $U_2^*$ are positive decomposition. Let $U_1$ and $U_2$ be positive harmonic  such that their
boundary functions on  $\partial \Omega$ are  $U_1^*$ and $U_2^*$.      We have $\|U^*\|_p ^p = \|U^*_1\|_p^p +
\|U ^* _2\|_p^p $. Since   $\|F  \|_p ^p=   \|F_1-F_2\|_p^p\le  2^{p}(\|F_1\|_p^p+\|F_2\|_p^p)$, it follows that
$\|F\|_p \le 2c_p \|U \|_p$.

Note that the constant $c_p$ in the Riesz  inequality  \eqref{EQ.R}    does not depend  on $\zeta_0\in
\Omega$,  although the  norm of a function  in the  Lumer's  Hardy space $(Lh)^p(\Omega)$ does. If $\Omega$
is a  simply connected  domain with at least two boundary points, this  is  expected, since the       group
of all conformal automorphisms  of the domain         $\Omega$ acts transitively on $\Omega$, i.e., for any
$\tilde{\zeta}_0\in\Omega$ there exists a  conformal automorphism   $\Phi$ of $\Omega$            such that
$\Phi(\tilde{\zeta}_0)  = \zeta_0$.    As we have already  said, the  mapping $\Phi$ induces  an  isometric
isomorphism of $(Lh)^p(\Omega)$ onto itself.  However, for multi-connected domains     it is  not true,  in
general, that the group  of all   conformal  automorphisms   acts  transitively on a domain.

In seventies,    Stout \cite[Theorem  IV.1]{STOUT.AJM}  proved   Riesz's  theorem  for Lumer's Hardy spaces
$(LH)^p(\Omega)$ on       $\mathcal {C}^2$-smooth domains $\Omega\subseteq \mathbf{C}^n$ (without a precise
constant in the Riesz inequality).    In this case there exists an integral representation  of  the Lumer's
norm of an    analytic function  that   is  used in order  to obtain the result.

Lumer's Hardy spaces $(Lh)^p (\Omega)$ and $(LH)^p (\Omega)$ on   domains $\Omega$ in $\mathbf{C}^n$    are
defined in a similar way  as in the one-dimensional case \cite{LUMER.PARIS}.        However, instead of the
harmonic     majorant  we have to use a pluriharmonic majorant, i.e., a function that is   locally the real
part of an  analytic function on  $\Omega$. Therefore, the  Lumer's Hardy space  $(Lh)^p (\Omega)$ contains
all pluriharmonic functions $U$ on $\Omega$  such that $ |U|^p$ has a pluriharmonic majorant    on $\Omega$.
The analytic Lumer's  Hardy space $(LH)^p (\Omega)$ is the subspace of  $(Lh)^p (\Omega)$     that consists
of all analytic functions.  The norm on $(Lh)^p (\Omega)$  may be introduced with respect to  any  $\zeta_0
\in\Omega$ using  the least  pluriharmonic majorant as  in  the ordinary case \eqref{EQ.LN}.

The proof given above works for any dimension.      Therefore, we have  the  dimension-free constant in the
Riesz inequality which is, moreover,    valid for all domains in $\mathbf{C}^n$.

\section{On vector-valued inequalities}

\subsection{}
Let $\tilde {D}$  be a domain in  $\mathbf{R}^n$ which   contains $0$ and has the property  that any intersection
with the plane containing $0$ is a same  domain  $D$. For  example $\tilde {D}$ may be the unit ball, the  closed
unit ball, or  the whole space.   We will use the following general principle for transferring  inequalities  for
functions defined  in the  domain  ${D}$   to inequalities         for functions in the domain  $\tilde{D}$.

\begin{lemma}
Let  $\mathcal{L}(x,y,s,t)$ and $\mathcal {R}(x,y,s,t)$ be  real-valued functions  defined for $x\ge 0$, $y\ge 0$
and $s\ge 0$, and   real $t$.  Let $\mathcal {C}$ be a class of mappings in $D$ such that
\begin{equation*}
\mathcal{L}(|f(\zeta)|, |f(\eta)|,|f(\zeta) - f(\eta)|,\left<f(\zeta),f(\eta)\right>)\le
\mathcal {R}(|\zeta|, |\eta|, |\zeta-\eta|,\left<\zeta,\eta\right>)
\end{equation*}
for  every          $f\in\mathcal{C}$ and $u\in D$ and $v\in D$.  Here $\left<\cdot,\cdot\right>$ stands for the
scalar product     in   $\mathbf{R}^n$.

Assume that $\tilde {\mathcal {C}}$ is a   class of mappings in the  domain  $\tilde{D}$ which  has the  property
that any  restriction of a function in $\tilde{C}$ composed by  a  projections        is a mapping  in the  class
$\mathcal {C}$). Then     the above  inequality ($\mathcal {L}\le \mathcal {R}$) remains   to  hold for  mappings
in    the  class  $\tilde {\mathcal {C}}$.
\end{lemma}

\begin{proof}
Choose  any mapping  $f\in\tilde{\mathcal{C}}$.    Let  $\zeta$ and $\eta$  be   fixed in the  domain $\tilde{D}$,
and denote $\Phi =  f(\zeta)$ and $\Psi=f(\eta)$. Our aim is to show that
\begin{equation*}
\mathcal {L}(|\Phi|,|\Psi|,|\Phi-\Psi|,\left<\Phi,\Psi\right>)
\le \mathcal{R}(|\zeta|,|\eta|,|\zeta-\eta|,\left<\zeta,\eta\right>).
\end{equation*}
Let  $D_{\zeta,\eta}$    be  the  intersection  of the domain $\tilde{D}$ and the plane $P _{\zeta,\eta}$,  which
contains $\zeta$, $0$ and $\eta$, and let  $P_{\Phi,\Psi}$   be the  plane which contains  $\Phi$, $0$ and $\Psi$.
Denote by $\Pi_{\Phi,\Psi}$    the orthogonal  projection   mapping of the   space  $\mathbf{R}^n$ onto the plane
$P_{\Phi,\Psi}$.   We will consider  the  mapping  which is  composition of  the restriction  of our mapping  $f$
on the domain  $D_{\zeta,\eta}$, denoted by $f|_{D_{\zeta,\eta}}$, and the above described projection  $\Pi_{\Phi,
\Psi}$. Therefore, let $\phi  =  \Pi _{\Phi,\Psi}  \circ  f|_{D_{\zeta,   \eta}}$. We have $\phi(\zeta)=\Phi$ and
$\phi(\eta) =\Psi$. Let $A_1$ and $A_2$ be orthogonal  which map   the   domain in the  plane $D$ onto  $D_{\zeta,
\eta}$  and  $D_{\Phi,\Psi}$,  respectively.   Since  $\phi: D_{\zeta,\eta}\rightarrow D_{\Phi,\Psi}$,   then the
mapping $A_2^{-1}\circ \phi\circ  A_1$ maps the domain $D$   into the plane. If we apply the assumed  fact at the
mapping  $\tilde{\phi}           =A_2^{-1}\circ\phi\circ A_1:D\rightarrow \mathbf{R}^2$ for points $\tilde{\zeta}
=A_1^{-1}(\zeta)$  and $\tilde{\eta}  =  A_1^{-1}(\eta)$, we obtain
\begin{equation*}
\mathcal {L}(|\tilde{\phi} (\tilde{\zeta})|, |\tilde{\phi} ( \tilde{\eta})|, |\tilde{\phi} (\tilde{\zeta}) - \tilde{\phi}
(\tilde{\eta})|,\left<\tilde{\phi} (\tilde{\zeta}),\tilde{\phi} (\tilde{\eta})\right>)\le
\mathcal {R}(|\tilde{\zeta}|, |\tilde {\eta}|,|\tilde {\zeta}-\tilde{\eta}|,\left<\tilde{\zeta},\tilde{\eta}\right>).
\end{equation*}
Since  $A_1^{-1}$ is distance and angle   preserving,  we obtain
\begin{equation*}
\mathcal {R}(|\tilde{\zeta}|, |\tilde {\eta}|, |\tilde {\zeta}-\tilde{\eta}|,\left<\tilde{\zeta},\tilde{\eta}\right>)  =
\mathcal {R}(|\zeta|,|\eta|,|\zeta-\eta|,\left<\zeta,\eta\right>).
\end{equation*}
Since $A_2$ is also  distance and angle  preserving, we also have
\begin{equation*}\begin{split}
&\mathcal {L}(|\tilde{\phi} (\tilde{\zeta})|, |\tilde{\phi} ( \tilde{\eta})|, |\tilde{\phi} (\tilde{\zeta})
 - \tilde{\phi} (\tilde{\eta})|,\left<\tilde{\phi} (\tilde{\zeta}),\tilde{\phi} (\tilde{\eta})\right>)
\\&= \mathcal {L}(|\phi (\zeta)|,|\phi(\eta)|,|\phi (\zeta)| - |\phi(\eta)|,\left<\phi (\zeta), \phi(\eta)\right>)
\\& = \mathcal {L}(|\Phi|,|\Psi|,|\Phi-\Psi|,\left<\Phi,\Psi\right>).
\end{split}\end{equation*}
Our inequality follows.
\end{proof}

\begin{corollary}
Let  $\mathcal{L}(x,y,s,t)$ and $\mathcal {R}(x,y,s,t)$ be  real-valued functions  defined for $x\ge 0$, $y\ge 0$
and $s\ge 0$, and   real $t$. Assume that we have inequality
\begin{equation*}
\mathcal{L}(|\zeta|, |\eta|,|\zeta- \eta|,\left<\zeta,\eta\right>)\le\mathcal {R}(|\zeta|,|\eta|, |\zeta-\eta|,\left<\zeta,\eta\right>)
\end{equation*}
for  every          $\zeta\in D\subseteq  \mathbf{R}^2$ and $\eta\in D \subseteq  \mathbf{R^2}$.

Then this inequality remains valid  for every  $\zeta\in  \tilde{D}\subseteq \mathbf{R}^n$ and  $\eta\in  \tilde{D}
\subseteq \mathbf{R}^n$.
\end{corollary}

\begin{proof}
We have just to apply the above lemma for the class which consists of only one function  which is identity.
\end{proof}

This corollary may be reformulated  as follows

\begin{corollary}
Let  $\mathcal{L}(x,y,s,t)$ and $\mathcal {R}(x,y,s,t)$ be  real-valued functions  defined for $x\ge 0$, $y\ge 0$
and $s\ge 0$, and   real $t$. Assume that we have inequality
\begin{equation*}
\mathcal{L}(|\zeta|, |\eta|,|\zeta \pm  \eta|,\angle(\zeta,\eta) )\le\mathcal {R}(|\zeta|,|\eta|, |\zeta\pm\eta|,\angle(\zeta,\eta))
\end{equation*}
for  every          $\zeta\in D\subseteq  \mathbf{R}^2$ and $\eta\in D \subseteq  \mathbf{R^2}$.

Then this inequality remains valid  for every  $\zeta\in  \tilde{D}\subseteq \mathbf{R}^n$ and  $\eta\in  \tilde{D}
\subseteq \mathbf{R}^n$.
\end{corollary}

This may be reformulates as:   if a certain inequality holds for plane vectors and depends  only on   norm of vectors
and angle between them,  then  the same inequality holds for vectors in the space. In other words if $\mathcal {F} (|z|, |w|,|z\pm w|,\angle (z,w))\ge 0$, then we have $\mathcal {F} (|\zeta|,
|\eta|,|\zeta\pm \eta|,\angle (\zeta,\eta))\ge 0$ for $\zeta\in\mathbf{R}^n$ and $\eta\in\mathbf{R}^n$.

\subsection{}
The corollary given  above will be applied here for the vector-valued  inequalities we need.  We prove  the following
inequality for   $z$ and $w$ in $\mathbf{C}^n$.
\begin{equation*}
(\|z\|^2 + \|w\|^2)^{\frac p2}\le C_{2,s}\|z+ \overline {w}\|^p - s(z,w)
\end{equation*}
where  $s(z,w)$ is plurisubharmonic in $\mathbf{C}^n\times \mathbf{C}^n$. Indeed, the Kalaj's  inequality states that
\begin{equation*}
( |z|^2 + |w|^2) ^{\frac p2}\le C_{p,2} |z+\overline{w}|^p- b_{p,2} \Re \left<z,\overline {w}\right>^{p/2}
\end{equation*}
for $z\in \mathbf{C}$ and $w\in \mathbf{C}$  and  for $1<p<2$, where $b_{p,2}$ is a positive constant.   From this
inequality we  deduce   the vector-valued one:
\begin{equation*}
( \|z\|^2+\|w\|^2) ^{\frac p2}\le C_{p,2} \|z+\overline{w}\|^p - b_p \|z\|^{\frac p2}\|w\|^{\frac p2} \cos \frac p2 \angle(z, w).
\end{equation*}
Therefore, we can take   $s(z,w) =  b_p \|z\|^{\frac p2}\|w\|^{\frac p2} \cos \frac p2 \angle(z, w)$.

Hollenbeck and Verbitsky proved the following  inequality
\begin{equation*}
\max\{ |z|^p,|w|^p\}\le C_{p,\infty} |z+\overline{w}|- b_{p,\infty} \Re (zw)^{\frac p2}
\end{equation*}
for $1<p<2$.  Since this inequality depends only on $|z|$ and $|w|$ and the angle between them, we may conclude that
it holds for  $z\in\mathbf{C}^n$ and $w\in\mathbf{C}^n$.  The vector--valued Verbitsky inequality may be  written as
\begin{equation*}
\max\{ \|z\|^p,\|w\|^p\}\le  C_{p,\infty} \|z+\overline{w}\|- b_p \Re \left<z,\overline {w}\right>^{p/2}
\end{equation*}
here $\left<\cdot,\cdot\right>$ is the inner  product in $\mathbf{C}^n$.
Inequalities proved in this paper for any $0<s<2$ may be transferred for $z\in\mathbf{C}^n$ and $w\in \mathbf{C}^n$.


\begin{thebibliography}{10}

\bibitem{DUREN.BOOK}
P. Duren,  \textit{Theory of $H^p$ spaces}, Academic Press,  New York and London, 1970.

\bibitem{FISHER.BOOK}
S.D. Fisher,  \textit{Function Theory on Planar Domains}, Wiley Interscience, New York,  1983.

\bibitem{HV.JFA}
B. Hollenbeck and I.E. Verbitsky,  \textit{Best constants for the Riesz projection},  J. Funct. Anal. \textbf{175} (2000), 370--392.

\bibitem{HV.OTAA}
B. Hollenbeck and I.E. Verbitsky, \textit{Best constant inequalities involving the analytic and co-analytic projection},
Oper. Theory   Adv. Appl. \textbf{202} (2010),   285--295.

\bibitem{GARNETT.BOOK}
J. B. Garnett, \textit{Bounded Analytic Functions},  Springer, New York,  2007.

\bibitem{GRAFAKOS.MRL}
L. Grafakos,  \textit{Best bounds for the Hilbert transform on $L^p(R^1)$},  Math. Res. Lett. \textbf{4} (1997),  469--471.

\bibitem{KALAJ.TAMS}
D. Kalaj, \textit{On Riesz type inequalities for harmonic mappings on the unit disk}, Trans. Amer. Math. Soc. \textbf{372} (2019), 4031--4051.

\bibitem{KALAJ.MARKOVIC.MS}
D. Kalaj and M. Markovi\'{c}, \textit{Norm of the Bergman projection}, Math. Scand. \textbf{115} (2014),  143--160.

\bibitem{LUMER.PARIS}
G. Lumer,  \textit{Espaces de Hardy en plusieurs variables complexes},  C. R. Acad. Sci. Paris S\'{e}r. A-B.  \textbf{273} (1971),  151--154.

\bibitem{LUMER-NAIM}
L. Lumer-Na\"{\i}m,  \textit{Hp spaces  of  harmonic functions},   Ann. Inst. Fourier  \textbf{17} (1967), 425--469.

\bibitem{MARKOVIC.AMM}
M. Markovi\'{c}, \textit{Riesz's theorem  for Lumer's Hardy spaces},  American Math. Monthly \textbf{127} (2020), 452--455.

\bibitem{MELENTIJEVIC.PHD}
P. Melentijevi\'{c},  \textit{Estimates of gradients and operator norm estimates in harmonic function theory},
PhD thsis, Belgrade 2018.

\bibitem{PAVLOVIC.PAMS}
M. Pavlovi\'{c},  \textit{A Short Proof of an  Inequality of Littlewood and Paley},  Proc. Amer. Math. Soc. \textbf{134} (2006), 3625--3627

\bibitem{PAVLOVIC.BOOK}
M. Pavlovi\'{c}, \textit{Function theory in the unit disk}, DeGruter, 2014.

\bibitem{PICHORIDES.STUDIA}
S.K. Pichorides, \textit{On the best values of the constants in the theorems of M. Riesz, Zygmund and Kolmogorov},
Stud.  Math. \textbf{44} (1972), 165--179.

\bibitem{RUDIN.BOOK}
W. Rudin, \textit{Real and Complex Analysis},  New York,  McGraw--Hill, 1987.

\bibitem{RUDIN.TAMS}
W. Rudin,  \textit{Analytic functions of class Hp}, Trans. Amer. Math. Soc.  \textbf{78} (1955),  46--66.

\bibitem{STOUT.AJM}
E.L. Stout, \textit{Hp-functions on strictly pseudoconvex domains},  Amer.   J.   Math.  \textbf{98} (1976),  821--852.

\bibitem{VERBITSKY.ISSLED}
I.E. Verbitsky,  \textit{Estimate of the norm of a function in a Hardy Space in terms of the norms of its real and imaginary parts},
Amer. Math. Soc. Transl.  \textbf{24} (1984),  11--15.

\end{thebibliography}
\end{document}